\documentclass[11pt]{amsart}
\usepackage{amsmath,amssymb,amsthm}
\usepackage[latin1]{inputenc}
\usepackage{tabularx,multicol,array}
\usepackage{graphicx,float,psfrag}
 \usepackage{stmaryrd,mathrsfs}
\usepackage{url}
\usepackage{tikz}
\usepackage{caption}
\usepackage{subcaption}
\usepackage{mathrsfs}
\usepackage{accents}
\usepackage{mathtools}
\newcommand{\reff}[1]{(\ref{#1})}

\theoremstyle{plain}
\newtheorem{theo}{Theorem}[section]
\newtheorem{theo*}{Theorem}

\newtheorem{prop}[theo]{Proposition}
\newtheorem{lem}[theo]{Lemma}

\theoremstyle{remark}
\newtheorem{rem}[theo]{Remark}

\newcommand{\cb}{{\mathcal B}}
\newcommand{\cc}{{\mathcal C}}

\newcommand{\cf}{{\mathcal F}}
\newcommand{\cg}{{\mathcal G}}
\newcommand{\ch}{{\mathcal H}}
\newcommand{\ci}{{\mathcal I}}

\newcommand{\ck}{{\mathcal K}}

\newcommand{\cp}{{\mathcal P}}

\newcommand{\cu}{{\mathcal U}}

\newcommand{\E}{{\mathbb E}}

\newcommand{\N}{{\mathbb N}}
\renewcommand{\P}{{\mathbb P}}

\newcommand{\R}{{\mathbb R}}

\newcommand{\Z}{{\mathbb Z}}

\newcommand{\rP}{{\rm P}}
\newcommand{\rE}{{\rm E}}

\newcommand{\ind}{{\bf 1}}

\newcommand{\argmin}{{\rm argmin}\;}
\newcommand{\argmax}{{\rm argmax}\;}

\newcommand{\norm}[1]{\mathop{\parallel\! #1 \! \parallel}\nolimits}
\newcommand{\val}[1]{\mathop{\left| #1 \right|}\nolimits}
\newcommand{\inv}[1]{\mathop{\frac{1}{ #1}}\nolimits}
\newcommand{\expp}[1]{\mathop {\mathrm{e}^{ #1}}}

\newcommand{\Var}{{\rm Var}\;}
\newcommand{\kl}[2]{\mathop{D\left(#1 \| #2 \right)}}

\newcommand{\sca}[2]{\mathop{\left\langle #1 , #2 \right\rangle}}
\newcommand{\pen}{{\rm pen}}
\newcommand{\var}{{\rm Var}\;}

\newcommand{\Cov}{{\rm Cov}\;}

\newcommand{\tr}[1]{\mathop{{\rm tr}\left( #1 \right)}\nolimits}

\renewcommand{\phi}{\varphi}
\renewcommand{\epsilon}{\varepsilon}

\newcommand{\fD}{\hat f_*^D}
\newcommand{\tD}{\hat t_*^D}
\newcommand{\yD}{\hat \psi^D_*}
\newcommand{\fS}{\hat f_*^S}
\newcommand{\gS}{\hat g_*^S}

\title{Optimal exponential bounds for aggregation of estimators for the Kullback-Leibler loss}
\date{\today}

 \author{Cristina Butucea}
 \address{
 Cristina Butucea,
 Universit\'{e} Paris-Est, LAMA (UPE-MLV), 77455 Marne La Vall\'{e}e, France.}
 \email{cristina.butucea@univ-mlv.fr}

 \author{Jean-Fran\c{c}ois Delmas}
 \address{
 Jean-Fran\c{c}ois Delmas,
 Universit\'{e} Paris-Est, CERMICS (ENPC), 77455 Marne La Vall\'{e}e, France.}
 \email{delmas@cermics.enpc.fr}

 \author{Anne Dutfoy}
 \address{
 Anne Dutfoy, 
 EDF Research \& Development, Industrial Risk Management Department, 92141 Clamart Cedex, France.}
 \email{anne.dutfoy@edf.fr}

 \author{Richard Fischer}
 \address{
 Richard Fischer, 
 Universit\'{e} Paris-Est, CERMICS (ENPC), 77455 Marne La Vall\'{e}e, France\\
 EDF Research \& Development, Industrial Risk Management Department, 92141 Clamart Cedex, France.}
 \email{fischerr@cermics.enpc.fr}

\begin{document}

\thanks{This work is partially supported by the French ``Agence Nationale de
 la Recherche'',CIFRE n$^{\circ}$ 1531/2012, and by EDF Research \& Development, Industrial Risk Management Department}



 \begin{abstract} 
 We study the problem of model selection type aggregation with respect to the Kullback-Leibler divergence for various probabilistic models. Rather than considering a convex combination of the initial estimators $f_1, \hdots, f_N$, our aggregation procedures rely on the convex combination of the logarithms of these functions. The first method is designed for probability density estimation as it gives an aggregate estimator that is also a proper density function, whereas the second method concerns spectral density estimation and has no such mass-conserving feature. We select the aggregation weights based on a penalized maximum likelihood criterion.  We give sharp oracle inequalities that hold with high probability, with a remainder term that is decomposed into a bias and a variance part. We also show the optimality of the remainder terms by providing the corresponding lower bound results.   
 
\end{abstract} 
 
\maketitle

 \section{Introduction} \label{sec:intro}

   The pure aggregation framework with deterministic estimators was first established in \cite{nemirovski2000topics} for nonparametric regression with random design.  Given $N$ estimators $f_k, 1 \leq k \leq N$ and a sample $X=(X_1, \hdots, X_n)$ from  the model $f$, the problem is to find an aggregated estimate $\hat{f}$ which performs nearly as well as the best  $f_\mu$, $\mu \in \cu$, where:
   \[
      f_\mu = \sum_{k=1}^N \mu_k f_k,
   \]
   and $\cu$ is a certain subset of $\R^N$ (we assume that linear combinations of the estimators are valid candidates).  The performance of the estimator is measured by a loss function $L$. Common loss functions include $L^p$ distance (with $p=2$ in most cases), Kullback-Leibler or other divergences, Hellinger distance, etc. The aggregation problem can be formulated as follows: find an aggregate estimator $\hat{f}$ such that for some $C \geq 1$ constant, $\hat{f}$ satisfies an oracle inequality in expectation, i.e.:
   \begin{equation} \label{eq:oracle_risk}
          \mathbb{E} \left[ L(f,\hat{f}) \right] \leq C \min_{\mu \in \cu} L(f,f_\mu) +  R_{n,N},
   \end{equation}
   or in deviation, i.e. for $\varepsilon > 0$ we have with probability greater than $1-\varepsilon$:
   \begin{equation} \label{eq:oracle_deviation}
           L(f,\hat{f}) \leq C \min_{\mu \in \cu} L(f,f_\mu) +  R_{n,N,\varepsilon},
   \end{equation}
   with remainder terms $R_{n,N}$ and $R_{n,N,\varepsilon}$ which do not depend on $f$ or $f_k, 1 \leq k \leq N$. If $C=1$, then the oracle inequality is sharp. 
   
   Three types of problems were identified depending on the choice of
   $\cu$.  In the model selection problem, the estimator mimics the best
   estimator amongst $f_1, \hdots, f_N$, that is $\cu= \{ e_k, 1 \leq k
   \leq N\}$, with $e_k = (\mu_j, 1 \leq j \leq N) \in \R^N$ the unit
   vector in direction $k$ given by $\mu_j=\ind_{\{j=k\}}$. In the
   convex aggregation problem, $f_\mu$ are the convex combinations of
   $f_k, 1 \leq k \leq N$, i.e. $\cu = \Lambda^+
   \subset \R^N$ with:
\begin{equation}
      \label{eq:def-lambda+}
          \Lambda^+ =\{ \mu=(\mu_k, 1 \leq k \leq N) \in \R^N, \mu_k
          \geq 0 \text{ and } \sum_{1 \leq k \leq N} \mu_k=1\}. 
\end{equation}  
   Finally in the linear aggregation problem we take $\cu = \R^N$,  the entire linear span of the initial estimators.    
   
    Early papers usually consider the $L^2$ loss in expectation as in \reff{eq:oracle_risk}. For the regression model with random design, optimal bounds for the  $L^2$ loss in expectation for model selection aggregation was considered in \cite{yang2000combining} and \cite{wegkamp2003model}, for convex aggregation in \cite{juditsky2000functional} with improved results for large $N$ in \cite{yang2004aggregating}, and for linear aggregation in \cite{tsybakov2003optimal}.  These results were extended to the case of regression with fixed design for the model selection aggregation in \cite{dalalyan2007aggregation} and \cite{dalalyan2008aggregation},  and for affine estimators in the convex aggregation problem in \cite{dalalyan2012sharp}. A unified aggregation procedure which achieves near optimal loss for all three problems simultaneously was proposed in \cite{bunea2007aggregation}. 
   
   For density estimation, early results include \cite{catoni1999universal} and \cite{yang2000mixing} which independently considered the model selection aggregation under the  Kullback-Leibler loss in expectaion. They introduced the progressive mixture method to give a series of estimators which verify oracle inequalities with optimal remainder terms. This method was later generalized as the mirror averaging algorithm in \cite{juditsky2008learning} and applied to various problems. Corresponding lower bounds which ensure the optimality of this procedure was shown in \cite{lecue2006lower}. The convex and linear aggregation problems for densities under the $L^2$ loss in expectation were considered in \cite{rigollet2007linear}.  
   
   While a lot of papers considered the expected value of the loss, relatively few papers address the question of optimality in deviation, that is with high probability as in \reff{eq:oracle_deviation}. For the regression problem with random design, \cite{audibert2008progressive} shows that the progressive mixture method is deviation sub-optimal for the model selection aggregation problem, and proposes a new algorithm which is optimal for the $L^2$ loss in deviation and expectation as well. Another deviation optimal method based on sample splitting and empirical risk minimization on a restricted domain was proposed in \cite{lecue2009aggregation}. For the fixed design regression setting, \cite{rigollet2012kullback} considers all three aggregation problems in the context of generalized linear models and gives  constrained likelihood maximization methods which are optimal in both expectation and deviation with respect to the Kullback-Leibler loss. More recently, \cite{dai2012deviation} extends the results of  
   \cite{rigollet2012kullback} for model selection by introducing the $Q$-aggregation method and giving a greedy algorithm which produces a sparse aggregate achieving the optimal rate in deviation for the $L^2$ loss. More general properties of this method applied to other aggregation problems as well are discussed in \cite{dai2014aggregation}. 
   
   For the density estimation, optimal bounds in deviation with respect to the $L^2$ loss for model selection aggregation are given in \cite{bellec2014optimal}. The author gives a non-asymptotic sharp oracle inequality under the assumption that $f$ and the estimators $f_k , 1 \leq k \leq N$ are bounded, and shows the optimality of the remainder term by providing the corresponding lower bounds as well. The penalized empirical risk minimization procedure introduced in \cite{bellec2014optimal} inspired our current work. Here, we consider a more general framework which incorporates, as a special case, the density estimation problem. Moreover, we give results in deviation for the Kullback-Leibler loss instead of the $L^2$ loss considered in \cite{bellec2014optimal}.
   
   Linear aggregation of lag window spectral density estimators with  $L^2$ loss  was studied in \cite{chang2014aggregation}. The method we propose is more general as it can be applied to any set of estimators $f_k$, $1 \leq k \leq N$, not only kernel estimators. However we consider the model selection problem, which is  weaker than the linear aggregation problem. Also, this paper concerns optimal bounds in deviation for the Kullback-Leibler loss instead of the $L^2$ loss in expectaion.

   We now present our main contributions. We propose aggregation schemes for the  estimation of probability densities on $\R^d$ and the estimation of spectral densities of stationary Gaussian processes.
    We consider model selection type aggregation for the Kullback-Leibler loss in deviation. For positive, integrable functions $p,q$, let $\kl{p}{q}$ denote the generalized Kullback-Leibler divergence given by: 
\begin{equation}
   \label{eq:KL}
      \kl{p}{q} = \int p \log(p/q) - \int p + \int q.
\end{equation}
   This is a Bregman-divergence, therefore $\kl{p}{q}$ is non-negative and $\kl{p}{q} =0$ if and only if a.e. $p=q$.  The Kullback-Leibler loss of an estimator $\hat{f}$ is given by $D\,(f||\hat{f})$. For initial estimators $f_k, 1 \leq k \leq N$, the aggregate estimator $\hat{f}$  verifies the following sharp oracle inequality for every $f$ belonging to a large class of functions $\cf$, with probability greater than $1-\exp(-x)$ for all $x > 0$:
   \begin{equation} \label{eq:oracle}
      \kl{f}{\hat{f}} \leq  \min_{1\leq k \leq N} \kl{f}{f_k} +  R_{n,N,x}.
   \end{equation}
    We propose two methods of convex aggregation for non-negative estimators, see Propositions \ref{prop:aggreg_gen_D} and \ref{prop:aggreg_gen_D}. Contrary to the usual approach of giving an aggregate estimator which is a linear or convex combination of the initial estimators, we consider an aggregation based on a convex combination of the logarithms of these estimators. The {\it convex aggregate estimators} $\hat{f} = f^D_{\hat{\lambda}}$ and $\hat{f} = f^S_{\hat{\lambda}}$  with $\hat{\lambda} = \hat{\lambda}(X_1, \hdots, X_n) \in \Lambda^+$  maximizes a penalized maximum likelihood criterion. The exact form of the convex aggregates $f^D_{\hat{\lambda}}$ and $f^S_{\hat{\lambda}}$ will be precised in later sections for each setup.  
   
   The first method concerns estimators with a given total mass and produces an aggregate $f^D_{\hat{\lambda}}$ which has also the same total mass. This method is particularly adapted for density estimation as it provides an aggregate which is also proper density function. We use this method to propose an adaptive nonparametric density estimator for maximum entropy distributions of order statistics in \cite{butucea2016nonparametric}.   
   The second method, giving the aggregate $f^S_{\hat{\lambda}}$, does not have the mass conserving feature, but can be applied to a wider range of statistical estimation problems, in particular to spectral density estimation. We show that both procedures give an aggregate which verifies a sharp oracle inequality with a bias and a variance term.  When applied to density estimation, we obtain sharp oracle inequalities with the optimal remainder term of order $\log(N)/n$, that is we have \reff{eq:oracle} with:
   \[
       R_{n,N,x}= \beta\frac{\log(N)+x}{n},
   \]
    with $\beta$ depending only on the infinity norm of the logarithms of $f$ and $f_k, 1 \leq k \leq N$, see Theorem \ref{theo:aggreg_density}. In the case of spectral density estimation, we need to suppose a minimum of regularity for the logarithm of the true spectral density and the estimators. We require that the logarithms of the functions belong to the periodic Sobolev space $W_{r}$ with $r > 1/2$. We show that this also implies that the spectral densities itself belong to $W_r$. We obtain \reff{eq:oracle} with:
    \[
       R_{n,N,x}= \beta\frac{\log(N)+x}{n}+\frac{\alpha}{n},
    \]
    where $\beta$ and $\alpha$ constants which depend only on the regularity and the Sobolev norm of the logarithms of $f$ and $f_k, 1 \leq k \leq N$, see Theorem \ref{theo:aggreg_spectral}.
    
    To show the optimality in deviation of the aggregation procedures, we give the corresponding tight lower bounds as well, with the same remainder terms, see Propositions \ref{prop:aggreg_low} and \ref{prop:spectr_low}. This complements the results of \cite{lecue2006lower} and \cite{bellec2014optimal} obtained for the density estimation problem. In \cite{lecue2006lower}  the lower bound for the expected value of the Kullback-Leibler loss was shown with the same order for the remainder term, while in \cite{bellec2014optimal} similar results were obtained in deviation for the $L^2$ loss. 
    
    The rest of the paper is organised as follows. In Section \ref{sec:not} we introduce the notation and give the basic definitions used in the rest of the paper. We present the two types of convex aggregation method for the logarithms in Sections \ref{sec:aggr_method_D} and \ref{sec:aggr_method_S}. For the model selection aggregation problem, we give a general sharp oracle inequality in deviation for the Kullback-Leibler loss for each method. In Section \ref{sec:app} we apply the methods for the probability density and the spectral density estimation problems. The results on the corresponding lower bounds can be found in Section \ref{sec:lower} for both problems.  We summarize the properties of Toeplitz matrices and periodic Sobolev spaces in the Appendix.

\section{Notations} \label{sec:not}

Let $\cb_+(\R^d)$, $d\geq 1$, be the set of non-negative measurable real
function  defined  on  $\R^d$ and   $h\in  \cb_+(\R^d)$  be  a  reference
probability density. For $f\in \cb_+(\R^d)$, we define:
\begin{equation}
   \label{eq:def-g}
g_f=\log(f/h),
\end{equation}
with the  convention that $\log(0/0)=0$. Notice  that we have $\norm{g_f}_\infty
<\infty $  if and only if   $f$ and $h$ have the same  support
$\ch=\{h>0\}$.  We  consider the subset $\cg$ of the
set of non-negative measurable functions with support $\ch=\{h>0\}$:
\[
      \cg= \{ f\in \cb_+(\R ^d);  \, \norm{g_f}_\infty < +\infty \}.
\]
For $f\in
\cg$, we set:
\begin{equation}
   \label{eq:def-m}
m_f=\int f, \quad \psi_f=-\int g_f \,h\quad \text{and} \quad
t_f=g_f+\psi_f,
\end{equation}
and   we   get  $\int t_f\,  h=0$ as well as the inequalities:
\begin{equation}
   \label{eq:majo-mty}
m_f\leq  \expp{\norm{g_f}_\infty}, \quad
|\psi_f|\leq \norm{g_f}_\infty, \quad
\norm{t_f}_\infty\leq 2\norm{g_f}_\infty\quad\text{and}\quad
\psi_f +\log(m_f)\leq  \norm{t_f}_\infty .
\end{equation}
Notice  that the  Kullback-Leibler divergence  $ \kl{f'}{f}$,  defined in
\reff{eq:KL}, is  finite for  any function  $f',f\in \cg$.   When
there  is no  confusion, we  shall write  $g$, $m$,  $\psi$ and  $t$ for
$g_f$, $m_f$, $\psi_f$ and $t_f$.

We consider  a probabilistic model $\cp=\{\rP_f;\,  f\in \cf(L)\}$, with
$\cf(L)$  a  subset  of  $\cg$  with  additional  constraints  (such  as
smoothness or integral condition) and $\rP_f$ a probability distribution
depending on  $f$. In  the sequel,  the model  $\rP_f$ corresponds  to a
sample   of  i.i.d.   random   variables  with   density  $f$   (Section
\ref{sec:aggr_method_D}) or a sample  from a stationary Gaussian process
with spectral density $f$  (Section \ref{sec:aggr_method_S}). Suppose we
have $(f_k, 1 \leq k \leq N)$,  which are $N$ distinct estimators of the
function $f\in \cf(L)$ such that  there exists $K>0$ (possibly different
from $L$) for which $f_k\in  \cf(K)$ for $1\leq k\leq N$, as  well as a sample
$X=(X_1,  \hdots, X_n)$,  $n \in  \N^*$ with  distribution $\rP_f$.   We
shall propose  two convex aggregation  estimator of $f$, based  on these
estimators   and  the   available  sample,   that  behaves,   with  high
probability, as well as the best initial estimator $f_{k^*}$ in terms of
the Kullback-Leibler divergence, where $k^*$ is defined as:
\begin{equation} \label{eq:def_k*_gen}
       k^* = \mathop{\argmin}_{1 \leq k \leq N} \kl{f}{f_{k}}.
\end{equation}

For $1 \leq k \leq N$,  we set  $g_k=g_{f_k}$, $m_k=m_{f_k}$, $\psi_k=\psi_{f_k}$ and
$t_k=t_{f_k}$. Notice that:
\begin{equation} \label{eq:exp_f}
      f=\exp(g ) \,h=\exp(t-\psi ) \,h  \quad \text{and}\quad
      f_k=\exp(g_k )\, h=\exp(t_k-\psi_k ) \,h. 
\end{equation}

We denote  by $I_n$  an integrable estimator  of the  function $f$  measurable with
respect to the sample $X=(X_1, \hdots, X_n)$. The estimator $I_n$ may be
a biased estimator of $f$. We note $\bar{f}_n$ the expected value of $I_n$:
     \[
       \bar{f}_n = \E[ I_n].
     \]

We   fix   some   additional   notation.   For   a  measurable  function
$p $ on $\R ^d$ and a measure $Q$ on  $\R^d$ (resp. a measurable
function $q$ on $\R^d$), we write  $\langle p , Q
\rangle=  \int p(x)Q(dx)$ (resp. $\langle p ,q
\rangle=  \int pq$)  when the integral is well defined. 
We shall consider the $L^2(h)$  norm given by $\norm{p}_{L^2(h)}  =
\left(\int p^2 h\right)^{1/2}$.
      
\section{Convex aggregation for the Kullback-Leibler divergence}  
\label{sec:aggr_method}
     
     In this section, we propose two convex aggregation methods, suited for models submitted to different type of constraints. First, we state non-asymptotic oracle inequalities for the Kullback-Leibler divergence in general form. Then, we derive more explicit non-asymptotic bounds for two applications: the probability density model and the spectral density of stationary Gaussian processes, respectively.

   \subsection{Aggregation procedures} 
   
   In this section, we describe the two aggregation methods of $f$ using
   the estimators $(f_k,  1\leq k\leq N)$.  The first one  is the convex
   aggregation of  the centered logarithm  $(t_k, 1\leq k\leq  N)$ which
   provides an aggregate estimator  $f^D_\lambda$.  This is particularly
   useful when considering density estimation, as the final estimator is
   also a density function.  The second one is the convex aggregation of
   the  logarithm $(g_k,  1\leq k\leq  N)$ which  provides an  aggregate
   estimator  $f^S_\lambda  $.  This  method is  suitable  for  spectral
   density estimation and it can be used for density estimation as well.

   \subsubsection{Density functions}
\label{sec:aggr_method_D}

In this Section, we shall  consider probability density function, but what follows can readily be
adapted to functions  with any given  total mass. Notice that if  $f\in \cg$
is a density, then  we get $\kl{h}{f}=\psi_f$,
which in turn 
implies that $\psi_f\geq 0$ that is, using also the last inequality of
\reff{eq:majo-mty}: 
\begin{equation}
   \label{eq:psi>0}
0\leq \psi_f\leq \norm{t_f}_\infty. 
\end{equation} 

We  want  to estimate  a  density  function $f  \in  \cg$  based on  the
estimators $f_k  \in \cg$ for $1  \leq k \leq  N$ which we assume  to be
probability density functions.  Recall the representation \reff{eq:exp_f} of $f$ and
$f_k$ with  $t=t_f$ and $t_k=t_{f_k}$.   For $\lambda \in \Lambda^+$  defined by
\reff{eq:def-lambda+}, we consider the aggregate estimator $f^D_\lambda$
given by the convex combination of $( t_k, 1 \leq k \leq N)$:
\[
      f^D_\lambda = \exp\left( t_\lambda-\psi_\lambda\right)\, h
      \quad\text{with}\quad 
      t_\lambda =\sum_{k=1}^N \lambda_k t_k
      \quad\text{and}\quad \psi_\lambda=\log\left(\int
        \expp{t_\lambda}\,  h \right). 
\]
Notice   that  $f^D_\lambda   $   is  a   density   function  and   that
$\norm{t_\lambda}_\infty \leq \max_{1\leq  k\leq N}
\norm{t_k}_\infty<+\infty  $, that is $f^D_\lambda\in \cg$.
The Kullback-Leibler divergence for the  estimator $f^D_\lambda$ of $f$ is
given by:
\begin{equation}
   \label{eq:KL-ffD}
     \kl{f}{f^D_\lambda}=\int f \log\left(f/f^D_\lambda\right) =
     \sca{t-t_\lambda}{f} + (\psi_\lambda-\psi). 
\end{equation}
Minimizing  the   Kullback-Leibler  distance   is  thus   equivalent  to
maximizing $ \lambda \mapsto \sca{t_\lambda}{f} - \psi_\lambda$.  Notice
that  $\sca{t_\lambda}{f}$  is  linear  in $\lambda$  and  the  function
$\lambda \mapsto  \psi_\lambda$ is convex since  $\nabla^2 \psi_\lambda$
is the covariance matrix of the random vector $(t_k (Y_\lambda), 1\leq k
\leq   N)$  with   $Y_\lambda$  having   probability  density   function
$f^D_\lambda$. As $I_n$ is a non-negative  estimator of $f$ based on the
sample  $X=(X_1,   \ldots,  X_n)$,   we  estimate  the   scalar  product
$\sca{t_\lambda}{f}$   by  $\sca{t_\lambda}{I_{n}}$.    To  select   the
aggregation weights $\lambda$, we  consider on $\Lambda^+$ the penalized
empirical criterion $H^D_n(\lambda)$ given by:
\begin{equation}\label{eq:def_Hn}
     H^D_n(\lambda) = \sca{t_\lambda}{I_{n}}  -\psi_\lambda - \inv{2}
     \pen^D(\lambda),  
\end{equation}
with penalty term:
\[
      \pen^D(\lambda) = \sum_{k=1}^N \lambda_k \kl{f^D_\lambda}{f_k} =
      \sum_{k=1}^N \lambda_k \psi_k - \psi_\lambda.  
\]

\begin{rem}\label{rem:1/2}
  The penalty term in \reff{eq:def_Hn} can be multiplied by any constant
  $\theta \in (0,1)$  instead of $1/2$.  The choice of  $1/2$ is optimal
  in  the sense  that  it  ensures that  the  constant $\exp(-6K)/4$  in
  \reff{eq:def_V_D}  of  Proposition  \ref{prop:aggreg_gen_D}  is  maximal,
  giving the sharpest result.
\end{rem}

The   penalty   term   is   always   non-negative   and   finite.    Let
$L^D_n(\lambda)= \sca{t_\lambda}{I_{n}} - \inv{2} \sum_{k=1}^N \lambda_k
\psi_k$. Notice that  $L^D_n(\lambda)$ is linear in  $\lambda$, and that
$H^D_n$ simplifies to:
\begin{equation} \label{eq:pen_alt_D}
        H^D_n(\lambda)= L^D_n(\lambda) - \inv{2} \psi_\lambda. 
\end{equation}

Lemma  \ref{lem:strong_conc_D} below  asserts  that the  function
$H^D_{n}$, defined  by \reff{eq:def_Hn},  admits a  unique  maximizer  on $\Lambda^+$  and  that it  is
strictly concave around this maximizer.
     
\begin{lem}\label{lem:strong_conc_D}
  Let $f$ and  $(f_k, 1\leq k\leq N)$ be density  functions, elements of
  $\cg$ such that $(t_k, 1\leq k\leq N)$ are linearly independent.  Then
  there exists a unique $\hat{\lambda}^D_* \in \Lambda^+$ such that:
\begin{equation} \label{eq:aggr_lambda_D}
        \hat{\lambda}^D_* = \mathop{\argmax}_{\lambda \in \Lambda^+}
        H^D_n(\lambda).  
\end{equation}
Furthermore, for all $\lambda \in \Lambda^+$,
we have: 
\begin{equation} \label{eq:strong_conc_D}
       H^D_{n}(\hat{\lambda}^D_*)-H^D_{n}(\lambda) \geq \frac{1}{2} \kl{f^D_{\hat{\lambda}^D_*}}{f^D_{\lambda}}.
\end{equation}
\end{lem}
    
\begin{proof}
  Consider  the form  \reff{eq:pen_alt_D}  of $H^D_n(\lambda)$.  Recall
  that the  function $\lambda  \mapsto L^D_n(\lambda)$ is  linear in
  $\lambda$  and  that  $\lambda \mapsto  \psi_\lambda$  is  convex.
  Notice that  $\nabla \psi_{\lambda}  =(\sca{t_k}{f^D_\lambda}, 1  \leq k
  \leq N)$. This implies that for all $\lambda,\lambda' \in \Lambda^+$:
\begin{align} \label{eq:psi_lambda}
\nonumber (\lambda-\lambda') \cdot \nabla \psi_{\lambda'} +
\kl{f^D_{\lambda'}}{f^D_{\lambda}} 
& = \sum_{k=1}^N
(\lambda_k-\lambda_k') \sca{ t_{k} }{f^D_{\lambda'}}  
 +\sca{t_{\lambda'}-t_{\lambda}}{f^D_{\lambda'}} +
 \psi_{\lambda}-\psi_{\lambda'} \\ 
& =  \psi_{\lambda}-\psi_{\lambda'}.
\end{align}
Since  $\psi_\lambda$  is  convex  and differentiable,  we  deduce  from
\reff{eq:pen_alt_D} that  $H^D_{n}$ is  concave and  differentiable.  We
also have by the linearity  of $L^D_n$ and \reff{eq:psi_lambda} that for
all $\lambda,\lambda' \in \Lambda^+$:
\begin{equation} \label{eq:Hlam-Hlam'_D}
        H^D_{n}(\lambda)-H^D_{n}(\lambda') = (\lambda-\lambda') \cdot
        \nabla H^D_{n}(\lambda') -  \inv{2}
        \kl{f^D_{\lambda'}}{f^D_\lambda}. 
\end{equation}
The concave  function $H^D_n$ on  a compact  set attains its  maximum at
some  points $\Lambda^*  \subset  \Lambda^+$.  For $\hat{\lambda}_*  \in
\Lambda^*$, we have for all $\lambda \in \Lambda^+$:
\begin{equation} \label{eq:H_n_max_D}
        (\lambda-\hat{\lambda}_*) \cdot \nabla H^D_{n}(\hat{\lambda}_*) \leq 0,
\end{equation}
see  for   example  Equation   4.21  of   \cite{boyd2004convex}.   Using
\reff{eq:Hlam-Hlam'_D}      with     $\lambda'=\hat{\lambda}_*$      and
\reff{eq:H_n_max_D},    we    get   \reff{eq:strong_conc_D}    .     Let
$\hat{\lambda}^1_*$   and  $   \hat{\lambda}^2_*$  be   elements  of   $
\Lambda^*$. Then by \reff{eq:strong_conc_D}, we have:
\[
      0=H^D_{n}(\hat{\lambda}^1_*)-H^D_{n}(\hat{\lambda}^2_*) \geq
      \frac{1}{2} \kl{f^D_{\hat{\lambda}^1_*}}{f^D_{\hat{\lambda}^2_*}}, 
\]
which       implies       that      a.e.       $f^D_{\hat{\lambda}^1_*}=
f^D_{\hat{\lambda}^2_*}$.  By the  linear independence  of $(t_k,  1\leq
k\leq N)$,  this gives  $\hat{\lambda}^1_* =  \hat{\lambda}^2_*$, giving
the uniqueness of the maximizer.
\end{proof}
    
Using   $\hat{\lambda}^D_*$ defined  in
\reff{eq:aggr_lambda_D}, we set:
\begin{equation}
   \label{eq:def-f*D}
\fD=f^D_{\hat{\lambda}^D_*}, \quad 
\tD=t_{\hat{\lambda}^D_*}\quad\text{and}\quad 
\yD=\psi_{\hat{\lambda}^D_*}. 
\end{equation}
We show that the convex aggregate estimator
$\fD$  verifies almost surely the
following non-asymptotic inequality with a bias and a variance term.

\begin{prop} 
\label{prop:aggreg_gen_D}  
Let $K>0$.   Let $f$ and  $(f_k, 1\leq  k\leq N)$ be  probability density functions,
elements  of  $\cg$  such  that  $(t_k, 1\leq  k\leq  N)$  are  linearly
independent and  $\max_{1\leq k\leq  N} \norm{t_k}_\infty \leq  K$.  Let
$X=(X_1, \hdots  , X_n)$ be  a sample from  the model $\rP_f$.   Then the following inequality holds:
\begin{equation*}
      \kl{f}{\fD} - \kl{f}{f_{k^*}} \leq B_n\left(\tD-t_{k^*}\right)+
      \mathop{\max}_{1 \leq k \leq N} V^D_n(e_k) ,      
\end{equation*}
with    the    functional    $B_n$    given   by,    for    $\ell    \in
L^\infty(\R)$:
\begin{equation} \label{eq:def_Bn}
       B_n(\ell) = \sca{\ell}{\bar{f}_n-f}. 
\end{equation}
and the function $V^D_n: \Lambda^+ \rightarrow \R$ given by:
\begin{equation}\label{eq:def_V_D}
     V^D_n(\lambda)= \sca{I_{n}-\bar{f}_n}{t_{\lambda}-t_{k^*}} -
     \frac{\expp{-6K}}{4} \sum_{k=1}^N \lambda_k \norm{t_k -
       t_{k^*}}^2_{L^2(h)}. 
\end{equation}
\end{prop}

\begin{proof}
Using \reff{eq:KL-ffD}, we get:
\[
 \kl{f}{\fD} - \kl{f}{f_{k^*}} 
 =  \sca{t_{k^*}-\tD}{f} + \yD  - \psi_{k^*}.
\]
By the  definition of  $k^*$, together with  $\pen^D(e_k)=0$ for  all $1
\leq  k \leq  N$  and the  strict  concavity \reff{eq:strong_conc_D}  of
$H^D_{n}$ at $\hat{\lambda}^D_*$ with $\lambda=e_{k^*}$, we get:
\begin{align*} 
\kl{f}{\fD} - \kl{f}{f_{k^*}}
& \leq   \sca{t_{k^*}-\tD}{f} + \yD - \psi_{k^*}  + H^D_{n}(\hat{\lambda}^D_*) - H^D_{n}(e_{k^*})
 -\inv{2}\kl{\fD}{f_{k^*}} \\ 
& =   \sca{\tD-t_{k^*}}{I_n- f}
-\inv{2}\kl{\fD}{f_{k^*}} - \inv{2}
\pen^D(\hat{\lambda}^D_*) \\ 
& =   B_n\left(\tD-t_{k^*}\right) + A^D_n,
\end{align*}
with:
\begin{equation} 
\label{eq:A_n_D}
A^D_n =  \sca{\tD-t_{k^*}}{I_{n}-\bar{f}_n}
-\inv{2}\kl{\fD}{f_{k^*}} - \inv{2} \sum_{k=1}^N
\hat{\lambda}^D_{*,k} \kl{\hat f^D_*}{f_{k}}. 
\end{equation}

We recall, see  Lemma 1 of \cite{barron1991approximation},  that for any
non-negative  integrable  functions $p$  and  $q$  on $\R^d$  satisfying
$\norm{\log(p/q)}_\infty < +\infty$, we have:
\begin{equation} \label{eq:BS1}
       \kl{p}{q} \geq \inv{2} \expp{-\norm{\log(p/q)}_\infty} \int p \left(\log(p/q)\right)^2 .
     \end{equation}
 We have:
\begin{align*}
 \kl{\fD}{f_k} 
& \geq \inv{2} \expp{-\norm{\log(\fD/ f_k)}_\infty}
\int \fD
\left(\log(\fD/f_k)\right)^2 \\ 
& \geq \inv{2}
\expp{-4K-\norm{\tD -\yD}_\infty}
\int h  \left(\log(\fD/f_k)\right)^2  \\ 
& \geq \inv{2} \expp{-6K}  \left( \norm{\tD -
    t_{k}}_{L^2(h)}^2 + (\yD-\psi_k)^2 \right) \\ 
& \geq \inv{2} \expp{-6K}  \norm{\tD-
  t_{k}}_{L^2(h)}^2, 
\end{align*}  
where we  used \reff{eq:BS1}  for the first  inequality, \reff{eq:psi>0}
for  the second,  and  \reff{eq:psi>0} as  well as  $\int  t_f h=0$  for
third. By using this  lower bound on $ \kl{\fD}{f_k} $  to both terms on
the right hand side of \reff{eq:A_n_D}, we get:
\begin{align*}
A^D_n   
& \leq  \sca{\tD-t_{k^*}}{I_{n}-\bar{f}_n} -
\frac{\expp{-6K}}{4}  \norm{\tD - t_{k^*}}_{L^2(h)}^2
- \frac{\expp{-6K}}{4}  \sum_{k=1}^N \hat{\lambda}^D_{*,k}
\norm{\tD - t_k}^2_{L^2(h)} \\ 
& = \sca{\tD-t_{k^*}}{I_{n}-\bar{f}_n}  -
\frac{\expp{-6K}}{4}  \sum_{k=1}^N \hat{\lambda}^D_{*,k}
\norm{t_k-t_{k^*} }^2_{L^2(h)} \\ 
& = V^D_n(\hat{\lambda}^D_*) ,
\end{align*}
where  the  first  equality  is   due  to  the  following  bias-variance
decomposition  equality  which  holds  for all  $\ell  \in  L^2(h)$  and
$\lambda \in \Lambda^+$:
\begin{equation} \label{eq:bias_var}
        \sum_{k=1}^N \lambda_k \norm{t_k - \ell}^2_{L^2(h)} =
        \norm{t_\lambda-\ell}^2_{L^2(h)} + \sum_{k=1}^N \lambda_k \norm{
          t_\lambda-t_k}^2_{L^2(h)}. 
\end{equation}
The  function $V^D_n$  is affine  in $\lambda$,  therefore it  takes its
maximum on $\Lambda^+$ at some $e_k$, $1 \leq k \leq N$, giving:
\begin{equation*} 
      \kl{f}{\fD} - \kl{f}{f_{k^*}} \leq B_n\left(\tD-t_{k^*}\right) +  \max_{1 \leq k \leq N} V^D_n(e_k).
     \end{equation*}
     This concludes the proof.     
     \end{proof}

\subsubsection{Non-negative functions} \label{sec:aggr_method_S}

In this  Section, we shall  consider non-negative functions. We  want to
estimate a  function $f \in \cg$  based on the estimators  $f_k \in \cg$
for $1  \leq k \leq N$.   Since most of  the proofs in this  Section are
similar to those  in Section \ref{sec:aggr_method_D}, we  only give them
when  there is  a substantial  new element.   Recall the  representation
\reff{eq:exp_f} of $f$  and $f_k$.  For $\lambda  \in \Lambda^+$ defined
by   \reff{eq:def-lambda+},   we   consider  the   aggregate   estimator
$f^D_\lambda$ given by  the convex aggregation of $( g_k,  1 \leq k \leq
N)$:
\begin{equation}
   \label{eq:aggr_f_S}
      f^S_\lambda = \exp\left( g_\lambda\right)\, h
      \quad\text{with}\quad 
      g_\lambda =\sum_{k=1}^N \lambda_k g_k.
\end{equation}
Notice   that  $\norm{g_\lambda}_\infty   \leq   \max_{1\leq  k\leq   N}
\norm{g_k}_\infty<+\infty  $,  that  is  $f^D_\lambda\in  \cg$.  We  set
$m_\lambda=m_{f^S_\lambda}$   the   integral   of   $f^S_\lambda$,   see
\reff{eq:def-m}.  The   Kullback-Leibler  distance  for   the  estimator
$f^S_\lambda$ of $f$ is given by:
\begin{equation}
\label{eq:kl_f_flambda_S}
     \kl{f}{f^S_\lambda}=\int f \log\left(f/f^S_\lambda \right) -
     m + m_\lambda=  \sca{g-g_\lambda}{f} - m + m_\lambda. 
\end{equation}   
Since  both  $g$  and  $g_\lambda$  are  bounded,  we   deduce  that
$\kl{f}{f_\lambda^S}  <   \infty$  for  all  $\lambda   \in  \Lambda^+$.
Minimization    of    the    Kullback-Leibler    distance    given    in
\reff{eq:kl_f_flambda_S} is therefore  equivalent to maximizing $\lambda
\mapsto     \sca{g_\lambda}{f}     -    m_\lambda$.      Notice     that
$\sca{g_\lambda}{f}$ is  linear in  $\lambda$ and the  function $\lambda
\mapsto  m_\lambda$ is  convex, since  the Hessian  matrix $\nabla^2
m_\lambda$ is given by: $\left[  \nabla^2 m_\lambda \right]_{i,j} = \int
g_i g_j f^S_\lambda $, which is positive-semidefinite.
As $I_n$ is a non-negative estimator of $f$ based on the sample $X=(X_1,
\ldots, X_n)$,  we estimate  the scalar product  $\sca{g_\lambda}{f}$ by
$\sca{g_\lambda}{I_{n}}$.   Here  we   select  the  aggregation  weights
$\lambda$ based  on the  penalized empirical  criterion $H^S_n(\lambda)$
given by:
\begin{equation}\label{eq:def_Hn_S}
     H^S_n(\lambda) = \sca{g_\lambda}{I_{n}}  -m_\lambda - \inv{2}
     \pen^S(\lambda),  
\end{equation}
    with the penalty term:
\[
      \pen^S(\lambda) = \sum_{k=1}^N \lambda_k \kl{f^S_\lambda}{f_k}=
      \sum_{k=1}^N \lambda_k  m_k -m_\lambda. 
\]
The choice of the factor $1/2$ for the penalty is justified by arguments
similar to  those given in  Remarks \ref{rem:1/2}.  The penalty  term is
always     non-negative     and    finite.      Let     $L^S_n(\lambda)=
\sca{g_\lambda}{I_{n}}  - \inv{2}  \sum_{k=1}^N \lambda_k  m_k$.  Notice
that  $L^S_n(\lambda)$   is  linear  in  $\lambda$,   and  that  $H^S_n$
simplifies to:
     \begin{equation} \label{eq:pen_alt_S}
        H^S_n(\lambda)= L^S_n(\lambda) - \inv{2} m_\lambda. 
     \end{equation}

Lemma  \ref{lem:strong_conc_S} below  asserts  that the  function
$H^S_{n}$  admits a  unique  maximizer  on $\Lambda^+$  and  that it  is
strictly concave around this maximizer.

\begin{lem}\label{lem:strong_conc_S}
  Let $f$  and $(f_k,  1\leq k\leq  N)$ be elements  of $\cg$  such that
  $(g_k, 1\leq  k\leq N)$  are linearly  independent.  Let  $H^S_{n}$ be
  defined   by   \reff{eq:def_Hn_S}.   Then  there   exists   a   unique
  $\hat{\lambda}^S_* \in \Lambda^+$ such that:
\begin{equation}
   \label{eq:aggr_lambda_S}
        \hat{\lambda}^S_* = \mathop{\argmax}_{\lambda \in \Lambda^+}
        H^S_n(\lambda). 
\end{equation}
Furthermore, for all $\lambda \in \Lambda^+$, we have:
\begin{equation}
\label{eq:strong_conc_S}
  H^S_{n}(\hat{\lambda}^S_*)-H^S_{n}(\lambda)      \geq      \frac{1}{2}
  \kl{f^S_{\hat{\lambda}^S_*}}{f^S_{\lambda}}.
\end{equation}
\end{lem}
    
\begin{proof}
Notice that for all $\lambda,\lambda' \in \Lambda^+$:
\begin{equation} \label{eq:m_lambda}
        m_{\lambda}-m_{\lambda'} = (\lambda-\lambda') \cdot \nabla
        m_{\lambda'} + \kl{f_{\lambda'}}{f_{\lambda}}. 
\end{equation}
The  proof    is  then  similar    to    the    proof    of    Lemma
\ref{lem:strong_conc_D}    using    \reff{eq:m_lambda}    instead    of
\reff{eq:psi_lambda}.
\end{proof}

Using   $\hat{\lambda}^S_*$ defined  in
\reff{eq:aggr_lambda_S}, we set:
\begin{equation}
   \label{eq:def-f*S}
\fS=f^S_{\hat{\lambda}^D_*}\quad\text{and}\quad 
\gS=g_{\hat{\lambda}^S_*}.
\end{equation}
We show that the convex aggregate estimator
$\fS$  verifies almost surely the
following non-asymptotic inequality with a bias and a variance term.

\begin{prop} 
\label{prop:aggreg_gen_S} 

Let $K >0$.  Let  $f$ and $(f_k, 1 \leq k \leq N)$  be elements of $\cg$
such  that  $(g_k,   1\leq  k\leq  N)$  are   linearly  independent  and
$\max_{1\leq k\leq N} \norm{g_k}_\infty \leq  K$.  Let $X=(X_1, \hdots ,
X_n)$ be a sample from the  model $\rP_f$. Then the following inequality
holds:
\begin{equation*}
      \kl{f}{\fS} - \kl{f}{f_{k^*}} \leq
      B_n\left(\gS-g_{k^*}\right)+ \mathop{\max}_{1
        \leq k \leq N} V^S_n(e_k) ,      
\end{equation*}
with the  functional $B_n$ given  by \reff{eq:def_Bn}, and  the function
$V^S_n: \Lambda^+ \rightarrow \R$ given by:
\[
     V^S_n(\lambda)= \sca{g_{\lambda}-g_{k^*}}{I_{n}-\bar{f}_n} -
     \frac{\expp{-3K}}{4} \sum_{k=1}^N \lambda_k \norm{g_k -
       g_{k^*}}^2_{L^2(h)}. 
\]
\end{prop}

\begin{proof}
Similarly to the proof of Proposition \ref{prop:aggreg_gen_D} we obtain that:
\[
      \kl{f}{\fS} - \kl{f}{f_{k^*}} \leq B_n\left(\gS-g_{k^*}\right) +
      A^S_n, 
\]
with:
\begin{equation} \label{eq:A_n_S}
        A^S_n =
        \sca{\gS-g_{k^*}}{I_{n}-\bar{f}_n}-\inv{2}\kl{\fS}{f_{k^*}}  -
        \inv{2} \sum_{k=1}^N \hat{\lambda}^S_{*,k} \kl{\fS}{f_{k}}. 
\end{equation}
Since $\norm{\log(\fS/f_k)}_\infty = \norm{g_{\hat{\lambda}^*}-g_k} \leq
2K$ for  $1 \leq k  \leq N$, we can  apply \reff{eq:BS1} with  $\fS$ and
$f_k$:
\begin{align} \label{eq:kl_geq_norm_S}
      \nonumber  \kl{\fS}{f_k} 
& \geq \inv{2} \expp{-\left\| \log(\fS/ f_k)\right\|_\infty} \int \fS
\left(\log(\fS/f_k)\right)^2 \\ 
\nonumber
& \geq \inv{2} \expp{-2K-\norm{\gS}_\infty} \int h  \left(\gS-g_k\right)^2  \\
& \geq \inv{2} \expp{-3K}  \norm{\gS - g_{k}}_{L^2(h)}^2,
\end{align} 
where in the  second and third inequalities we use that  $\norm{\gS}_\infty \leq 
\max_{1\leq    k\leq   N}    \norm{g_k}_\infty   \leq    K$.    Applying
\reff{eq:kl_geq_norm_S}  to  both  terms  on  the  right  hand  side  of
\reff{eq:A_n_S} gives:
     \begin{align*}
        A_n (\hat{\lambda}^S_*)    & \leq  \sca{\gS-g_{k^*}}{I_{n}-\bar{f}_n} - \frac{\expp{-3K}}{4}  \norm{\gS - g_{k^*}}_{L^2(h)}^2  - \frac{\expp{-3K}}{4}  \sum_{k=1}^N \hat{\lambda}^S_{*,k}  \norm{\gS - g_k}^2_{L^2(h)} \\
         & = \sca{\gS-g_{k^*}}{I_{n}-\bar{f}_n}  - \frac{\expp{-3K}}{4}  \sum_{k=1}^N \hat{\lambda}^S_{*,k}  \norm{g_k-g_{k^*} }^2_{L^2(h)} \\
         & = V^S_n(\hat{\lambda}^S_*) , 
\end{align*}
where we used \reff{eq:bias_var} for  the second equality.  The function
$V^S_n$  is affine  in  $\lambda$,  therefore it  takes  its maximum  on
$\Lambda^+$ at some $e_k$, $1 \leq k \leq N$, giving:
\[
       \kl{f}{\fS} - \kl{f}{f_{k^*}}  \leq B_n\left(\gS-g_{k^*}\right) +
       \max_{1 \leq k \leq N} V^S_n(e_k). 
\]
This concludes the proof.      \end{proof}

\subsection{Applications} \label{sec:app}
     
In  this   section  we   apply  the   methods  established   in  Section
\ref{sec:aggr_method_D}  and \ref{sec:aggr_method_S}  to the  problem of
density estimation  and spectral  density estimation,  respectively.  By
construction,     the     aggregate     $f^D_\lambda$     of     Section
\ref{sec:aggr_method_D}  is  more  adapted for  the  density  estimation
problem  as it  produces a  proper density  function.  For  the spectral
density estimation problem, the aggregate $f^S_\lambda$ will provide the
correct results.
     
\subsubsection{Probability density estimation} \label{sec:app_D}

We  consider the
following subset of probability density functions, for $L>0$:
\[
 \cf^D(L)=\{f\in \cg; \, \norm{t_f}_\infty \leq L \text{ and } m_f=1\}.
\]
The  model $\{\rP_f,  f \in  \cf^D(L)\}$ corresponds  to i.i.d.   random
sampling  from a  probability density  $f \in  \cf^D(L) $,  that is  the
random variable $X=(X_1, \hdots.  X_n)$ has density $ f^{\otimes n}(x) =
\prod_{i=1}^n f(x_i)$, with $ x = (x_1, \hdots, x_n) \in (\R^d )^n$.  We
estimate the  probability measure $f(y)dy$ by  the empirical probability
measure $I_n(dy)$ given by:
\[
       I_n(dy)= \inv{n} \sum_{i=1}^n \delta_{X_i}(dy),
\]
where $\delta_y$ is  the Dirac measure at $y \in
\R^d$. Notice that $I_n$ is an unbiased estimator of $f$:
\[
       f(y)dy=\E [I_n (dy)]   \quad \quad \text{ for } y =  \R^d .
\]

In  the  following  Theorem,  we  give  a  sharp  non-asymptotic  oracle
inequality in  probability for  the aggregation  procedure $\fD$  with a
remainder   term   of   order   $\log(N)/n$.   We   prove   in   Section
\ref{sec:lower_D} the  lower bound  giving that  this remainder  term is
optimal.
     
\begin{theo} \label{theo:aggreg_density}
Let $L,K>0$.   Let $f\in \cf^D(L)$  and $(f_k, 1  \leq k \leq  N)$ be
elements of $\cf^D(K)$ such that $(t_k, 1  \leq k \leq  N)$ are linearly
independent.  Let $X=(X_1, \hdots , X_n)$
  be an i.i.d. sample from $f$. Let $\fD$ be given by
  \reff{eq:def-f*D}. 
     Then for any $x>0$ we have with probability greater than $1-\exp(-x)$:
    \begin{equation*}
      \kl{f}{\fD} -  \kl{f}{f_{k^*}} \leq  \frac{\beta (\log(N)+x)}{n},     
    \end{equation*}
     with $\beta= 2\exp(6K+2L)+4K/3$.
\end{theo}

\begin{proof}
By Proposition \ref{prop:aggreg_gen_D}, we have that:
\begin{equation}  \label{eq:theo_density}
        \kl{f}{\fD} -  \kl{f}{f_{k^*}}  \leq
        B_n\left(\tD-t_{k^*}\right)+ \mathop{\max}_{1 \leq k \leq N}
        V^D_n(e_k). 
\end{equation}
Since $I_n(dy)$ is  an unbiased  estimator of $f(y)dy$,  we get $  B_n\left(\tD-t_{k^*}\right) =0$.
Notice that 
\begin{equation} \label{eq:bernstein}
      \P\left( V^D_n(e_k) \geq \frac{\beta(\log(N)+ x)}{n} \right) \leq
      \frac{\expp{-x}}{N} \quad \text{for all $1 \leq k \leq N$}, 
     \end{equation}
implies 
\[     
        \P\left( \max_{1 \leq k \leq N} V^D_n(e_k)  \geq \frac{\beta
            (\log(N)+x)}{n} \right) \leq  \expp{-x}, 
\] 
which will provide a  control of the second  term on the right  hand side of
\reff{eq:theo_density}.  Thus,  the  proof of  the  theorem  will  be
complete as soon as \reff{eq:bernstein} is proved.

To  prove \reff{eq:bernstein},  we use  the concentration  inequality of
Proposition 5.3 in \cite{bellec2014optimal}  which states that for $Y_1,
\hdots,  Y_n$ independent  random variables  with finite  variances such
that $\val{Y_i-\E Y_i}  \leq b$ for all  $1 \leq i \leq n$,  we have for
all $u>0$ and $a>0$:
     \begin{equation} \label{eq:bellec}
       \P\left( \inv{n}\sum_{i=1}^n \left(Y_i - \E Y_i - a \var Y_i\right) > \left(\inv{2a}+\frac{b}{3}\right)\frac{u}{n} \right) \leq \expp{-u}.
     \end{equation}
     Let us choose $Y_i= t_k(X_i)-t_{k^*}(X_i)$ for $1\leq i \leq
     n$. Then, since $f_k$ and $f_{k^*}$ belong to $\cf^D(K)$,
      we have $\val{Y_i-\E Y_i} \leq 4K$, and:
\begin{equation} \label{eq:var_maj}
       \var Y_i \leq \int (t_k-t_{k^*})^2 f \leq \expp{2L}
       \norm{t_k-t_{k^*}}^2_{L^2(h)}. 
\end{equation}
Applying   \reff{eq:bellec}   with    $a=\exp(-6K-2L)/4$,   $b=4K$   and
$u=\log(N)+x$, we obtain:
\begin{align*}
\frac{\expp{-x}}{N} 
&\geq \P\left( \sca{t_k-t_{k^*}}{I_n-\bar{f}_n} -
  \frac{\expp{-6K-2L}}{4} \var Y_1  > \frac{\beta(\log(N)+x)}{n} \right)
\\ 
&\geq \P\left( \sca{t_k-t_{k^*}}{I_n-\bar{f}_n} - \frac{\expp{-6K}}{4}
  \norm{t_k-t_{k^*}}^2_{L^2(h)}  > \frac{\beta ( \log(N)+ x)}{n} \right)
\\ 
& =  \P\left( V^D_n(e_k) > \frac{\beta (\log(N)+x)}{n} \right),
\end{align*}
where the  second inequality  is due  to \reff{eq:var_maj}.  This proves
\reff{eq:bernstein} and completes the proof.
\end{proof}

\begin{rem} \label{rem:f_D_f_S_diff}
We can  also use the aggregation method
  of    Section   \ref{sec:aggr_method_S}    and consider the
  normalized estimator
  $\tilde{f}^S_{*}= \fS/m_{\hat{\lambda}^S_*}= f^D_{\hat{\lambda}^S_*}$,
which  is a proper  density function.   Notice that the
  optimal weights  $\hat{\lambda}^D_*$ (which defines $\fD$) and
  $\hat{\lambda}^S_*$ (which defines $\tilde{f}^S_{*}$) maximize
  different criteria. Indeed, according to \reff{eq:aggr_lambda_S}  the
  vector $\hat{\lambda}^S_*$ maximizes: 
\[
H_n^S(\lambda) = \sca{g_\lambda}{I_n} - \inv{2}m_\lambda -
\inv{2}\sum_{k=1}^N \lambda_k m_k = \sca{g_\lambda}{I_n} -
\inv{2}m_\lambda - \inv{2}, 
\]
and according to \reff{eq:aggr_lambda_D} the vector $\hat{\lambda}^D_*$ maximizes:
\begin{align*}
H_n^D(\lambda) 
 = \sca{t_\lambda}{I_n} - \inv{2}\psi_\lambda - \inv{2}\sum_{k=1}^N
\lambda_k \psi_k 
 = \sca{g_\lambda}{I_n} 
-\inv{2}\psi_\lambda + \inv{2}\sum_{k=1}^N \lambda_k \psi_k   = \sca{g_\lambda}{I_n}  - \inv{2} \log(m_\lambda) , 
     \end{align*}
 where we used the identity $g_\lambda=t_\lambda - \sum_{k=1}^N
 \lambda_k \psi_k$ for the second equality and the  equality
$\log( m_\lambda) = \log \left( \int \expp{t_\lambda - \sum_{k=1}^N
    \lambda_k \psi_k} h  \right) = \psi_\lambda-\sum_{k=1}^N \lambda_k
\psi_k $
  for the
third. 
\end{rem}

\subsubsection{Spectral density estimation} \label{sec:app_S}
    
In  this section  we  apply  the convex  aggregation  scheme of  Section
\ref{sec:aggr_method_S}  to spectral  density  estimation of  stationary
centered Gaussian sequences. Let $h = 1/(2\pi) \ind_{[-\pi,\pi]}$ be the
reference density and $(X_k)_{k\in \Z}$ be a  stationary, centered
  Gaussian sequence with  covariance 
$\gamma$ function defined as, for $j \in \Z$:
\[
      \gamma_j=\Cov(X_k,X_{k+j}).
\]
Notice  that  $\gamma_{-j}=\gamma_j$.  Then the  joint  distribution  of
$X=(X_1, \ldots, X_n)$ is a multivariate, centered Gaussian distribution
with  covariance  matrix  $\Sigma_n  \in   \R^{n  \times  n}$  given  by
$ [\Sigma_n]_{i,j}  = \gamma_{i-j}$ for  $1\leq i,j \leq n$.  Notice the
sequence $(\gamma_j)_{j\in \Z}$ is semi-definite positive.

    We make the following standard assumption on the covariance function $\gamma$:
\begin{equation} \label{eq:cond_gamma}
\sum_{j=0}^\infty \val{\gamma_j} = C_1 < +\infty  . 
\end{equation}
The spectral density $f$ associated  to the process is the 
even  function defined  on $[-\pi,\pi]$  whose Fourier  coefficients are
$\gamma_j$:
\[
f(x)= \sum_{j\in \Z} \frac{\gamma_j}{2\pi} \expp{ijx}= \frac{\gamma_0}{2\pi}
+ \inv{\pi} \sum_{j=1}^\infty \gamma_j \cos(jx). 
\]
The first condition in \reff{eq:cond_gamma}  ensures that the spectral density
is  well-defined,  continuous  and  bounded by  $C_1/\pi$.  It  is  also
even and non-negative as $(\gamma_j)_{j\in \Z}$ is semi-definite positive.  The
function   $f$  completely   characterizes  the  model     as:
\begin{equation}
   \label{eq:def-ga}
\gamma_j=\int_{-\pi}^\pi f(x) \expp{ijx} \, dx= \int_{-\pi}^\pi f(x)
\cos(jx) \, dx\quad\text{for $j\in \Z$.} 
\end{equation}

For $\ell\in L^1(h)$, we define the corresponding
Toeplitz $T_n(\ell)$ of size $n\times n$ by:
\[
[T_n(\ell)]_{j,k}=\inv{2\pi}\int_{-\pi}^\pi \ell(x) \expp{i(j-k)x} \, dx.
\]
Notice that $T_n(2\pi f)=\Sigma_n$. Some  properties of the Toeplitz
matrix $T_n(\ell)$ are collected in Section \ref{sec:toeplitz}. 

We choose the following estimator of $f$, for $x \in [-\pi,\pi]$:
\[
I_{n}(x) = \frac{\hat{\gamma}_0}{2\pi} + \inv{\pi} \sum_{j=1}^{n-1}
\hat{\gamma}_j \cos(jx), 
\]
with $(\hat{\gamma}_j,  0 \leq j  \leq n-1)$ the empirical  estimates of
the correlations $(\gamma_j, 1 \leq j \leq n-1)$:
\begin{equation} \label{eq:def_gamma_est}
  \hat{\gamma}_j = \inv{n} \sum_{i=1}^{n-j} X_i X_{i+j}. 
\end{equation}
The function $I_{n}$ is a biased estimator, where the bias is due to two
different  sources:  truncation of  the  infinite  sum  up to  $n$,  and
renormalization in \reff{eq:def_gamma_est} by  $n$ instead of $n-j$ (but
it  is asymptotically  unbiased as  $n$  goes to  infinity if  condition
\reff{eq:cond_gamma} is  satisfied). The  expected value  $\bar{f}_n$ of
$I_{n}$ is given by:
\[
\bar{f}_n(x) =  \sum_{|j|< n} \left(1- \frac{|j|}{n}\right) \frac{\gamma_j}{2\pi} \expp{jx}
=
\frac{\gamma_0}{2\pi} + \inv{\pi} \sum_{j=1}^{n-1}
\frac{(n-j)}{n} \gamma_j \cos(jx). 
\]

In order to be able  to apply Proposition \ref{prop:aggreg_gen_S}, we assume
that $f$ and  the estimators $f_1, \hdots, f_N$ of  $f$ belongs to $\cg$
(they are in particular positive and bounded) and are even functions. In
particular the  estimators $f_1, \hdots,  f_N$ and the  convex aggregate
estimator  $\fS$  defined  in   \reff{eq:def-f*S}  are  proper  spectral
densities of stationary Gaussian sequences.

\begin{rem}
  By  choosing  $h  =   1/(2\pi)  \ind_{[-\pi,\pi]}$,  we  restrict  our
  attention to spectral  densities that are bounded  away from $+\infty$
  and $0$, see \cite{m:drsts} and \cite{b:psdf} for the characterization
  of such  spectral densities.  Note  that we can apply  the aggregation
  procedure to  non even  functions $f_k$,  $1 \leq k  \leq N$,  but the
  resulting estimator  would not  be a proper  spectral density  in that
  case.
\end{rem}

To prove a  sharp oracle  inequality for the  spectral density
estimation, since  $I_n$ is a biased  estimator of $f$, we  shall assume
some   regularity on the  functions $f$  and $f_1, \hdots,  f_N$ in
order to be able to control the bias term. More precisely those
conditions will be Sobolev conditions on  their logarithm, that is on the
functions $g$ and $g_1, \hdots, g_N$ defined by \reff{eq:def-g}.

For            $\ell\in            L^2(h)$,           the corresponding 
Fourier coefficients are defined  for $k\in \Z$ by
$a_k=\inv{2\pi} \int_{-\pi}^\pi \expp{-ikx} \ell(x)\, dx$.
From     the    Fourier     series     theory,     we    deduce     that
$\sum_{k\in    \Z}    |a_k|^2    =\norm{\ell}^2_{L^2(h)}$    and    a.e.
$\ell(x)=\sum_{k\in    \Z}    a_k     \expp{ikx}$.     If    furthermore
$\sum_{k\in   \Z}  |a_k|$   is  finite,   then  $\ell$   is  continuous,
$\ell(x)=\sum_{k\in  \Z}  a_k \expp{ikx}$  for  $x\in  [-\pi, \pi]$  and
$\norm{\ell}_\infty \leq \sum_{k\in \Z} |a_k|$.

For $r> 0$, we
define the Sobolev norm $\norm{\ell}_{2,r}$ of $\ell$ as:
\[
\norm{\ell}_{2,r}^2 =\norm{\ell}^2_{L^2(h)}+  \{\ell\}_{2,r}^2 \quad\text{with}\quad
\{\ell\}_{2,r}^2=\sum_{k\in \Z}
|k|^{2r} |a_k|^2.
\]
The corresponding      Sobolev space is
defined by:
\[
W_r=\{\ell\in  L^2(h); \, \norm{\ell}_{2,r}<+\infty \}.
\]
For $r>1/2$,  we can bound the supremum norm of $\ell$ by its Sobolev
norm:
\begin{equation}
   \label{eq:sobolev-ineq}
\norm{\ell}_\infty \leq \sum_{k\in \Z} |a_k|
\leq  \cc_r \{\ell\}_{2,r}\leq \cc_r \norm{\ell}_{2,r},
\end{equation}
where we used   Cauchy-Schwarz inequality for the
second inequality with
\begin{equation} \label{eq:def_C_r}
\cc_r^2 = \sum_{k\in \Z^*} |k|^{-2r} < + \infty. 
\end{equation}

The proof of the following Lemma seems to be part of the folklore, but
since we didn't find a proper reference, we give it in   Section
\ref{sec:preuve-exp-S}. 
\begin{lem}
   \label{lem:reg}
   Let $r>1/2$, $K>0$. There exists a finite constant $C(r,K)$ such that
   for  any  $g\in  W_r$  with   $\norm{g}_{2,r}\leq  K$,  then  we  have
   $\norm{\exp(g)}_{2,r}\leq C(r,K)$.
\end{lem}

For $r>1/2$, we consider the following subset of functions:
\begin{equation}
\label{eq:def-FS}
\cf^S_r(L)=
\{f\in \cg: \, \norm{g_f}_{2,r}\leq L/\cc_r \text{ and
  $g_f$ even}\}.
\end{equation}
For $f\in \cf^S_r(L)$, we  deduce from \reff{eq:sobolev-ineq} that $g_f$
is continuous (and bounded by $L$).  This implies that $f$ is a positive,
continuous, even function and thus  a proper spectral density. Notice
that $2\pi\norm{f}_\infty \leq \exp(L)$ . We deduce  from \reff{eq:def-ga} that
$\gamma_k=\int_{-\pi}^\pi \expp{-ikx} f(x)\, dx$ and thus:
\[
\norm{f}_{2,r}^2= \frac{\gamma_0^2}{4\pi^2} + \inv{2\pi^2} \sum_{k=1}^{\infty }
(1+k^{2r}) \gamma_k^2.
\]
Thus Lemma \ref{lem:reg} and \reff{eq:sobolev-ineq} imply also that  the covariance function
associated to $f\in \cf^S_r(L)$ satisfies \reff{eq:cond_gamma}. We also get that 
$\sum_{j=1}^\infty j \gamma_j^2 < + \infty$, which is a standard assumption for 
spectral density estimation.

The following Theorem is the main result of this section. 
      
\begin{theo} \label{theo:aggreg_spectral}  Let $r > 1/2$,  $K,L>0$.  Let
  $f \in \cf_r^S(L)$ and $(f_k, 1\leq  k\leq N)$ be elements of $\cf_r^S(K)$
  such  that  $(g_k, 1\leq  k\leq  N)$  are linearly  independent.   Let
  $X=(X_1, \ldots, X_n)$  be a sample of a  stationary centered Gaussian
  sequence  with  spectral   density  $f$.   Let  $\fS$   be  given   by
  \reff{eq:aggr_f_S}.   Then  for any  $x>0$,  we have  with
  probability higher than $1-\exp(-x)$:
\begin{equation*}
\kl{f}{\fS} -  \kl{f}{f_{k^*}} \leq \frac{\beta (\log(N)+x)}{n}
+ \frac{\alpha}{n},      
\end{equation*}
with $\beta= 4 (K \expp{L} +  \expp{2L+3K})$ and $\alpha=4KC(r,L)/\cc_r$. 
\end{theo}

 \begin{rem}
   When the value  of $\gamma_0$ is given, we shall  use the aggregation
   method  of  Section  \ref{sec:aggr_method_D}  after  normalizing  the
   estimators  $f_k$, $1\leq  k \leq  N$  by dividing  $f_k$ with  $m_k=\int f_k$.
   The    final    estimator    of    $f$   would    take    the    form
   $\tilde{f}^D_{\hat{\lambda}^D_*} =  \gamma_0 f^D_{\hat{\lambda}^D_*}$
   and  verifies a  similar sharp  oracle inequality  as $\fS$  (that is
   without       the       term        $\alpha/n$       of       Theorem
   \ref{theo:aggreg_spectral}). When the value of $\gamma_0$ is unknown,
   it        could       be        estimated       empirically        by
   $\hat{\gamma_0}  =  \inv{n} \sum_{i=1}^n  X_i^2$.  Then we could  use
   $  \hat{\gamma}_0 f^D_{\hat{\lambda}^D_*}$  to estimate  $f$. However
   the empirical  estimation of $\gamma_0$  introduces an error  term of
   order $1/\sqrt{n}$,  which leads to  a suboptimal remainder  term for
   this aggregation method.
 \end{rem}

\begin{proof}
Using Proposition \ref{prop:aggreg_gen_S} and the notations defined there, we have that:
\begin{equation}  \label{eq:theo_spectral}
 \kl{f}{\fS} -  \kl{f}{f_{k^*}}  \leq  B_n\left(\gS-g_{k^*}\right)+
 \mathop{\max}_{1 \leq k \leq N} V^S_n(e_k). 
\end{equation}
  
\subsubsection*{First step: Concentration inequality for $
\mathop{\max}_{1 \leq k \leq N} V^S_n(e_k)$.} 
We shall prove that 
\begin{equation} \label{eq:unif_exp_bound}
 \P\left( \max_{1 \leq k \leq N} V^S_n(e_k) \geq \frac{\beta
     (\log(N)+x)}{ n} \right) \leq  \expp{-x}. 
\end{equation}
It is enough to prove that for each $1 \leq k \leq N$:
\begin{equation} \label{eq:conc_ineq_spec}
\P\left( V^S_n(e_k) \geq \frac{\beta u}{n} \right) \leq  \expp{-u}.
\end{equation}
Indeed take $u=\log(N)+x$ and  the union bound over $1 \leq k \leq N$ to
deduce \reff{eq:unif_exp_bound} from \reff{eq:conc_ineq_spec}. 

The   end  of   this   first   step  is   devoted   to   the  proof   of
\reff{eq:conc_ineq_spec}.  Recall definition \reff{eq:toeplitz} of
Toeplitz matrices associated to Fourier coefficients.   We     express    the     scalar    product
$\sca{\ell}{I_{n}}$  for $\ell  \in \mathbb{L}^\infty([-\pi,\pi])$  in a
matrix form:
\begin{equation}
\label{eq:In_g}
 \sca{\ell}{I_{n}}  =   \inv{2 \pi n} \sum_{i=1}^n \sum_{j=1}^n X_i X_j
 \int_{-\pi}^\pi \ell(x) \cos((i-j)x) \, dx 
 = \inv{ n} X^T T_{n}(\ell) X.
\end{equation} 
We  have the  following  expression  of the  covariance  matrix of  $X$:
$\Sigma_n = 2\pi T_n(f)$.  Since $f$ is positive, we get that $\Sigma_n$
is positive-definite.  Set $\xi=\Sigma_n^{-1/2}  X$ so  that $\xi$  is a
centered $n$-dimensional Gaussian vector  whose covariance matrix is the
$n$-dimensional  identity  matrix.   By  taking the  expected  value  in
\reff{eq:In_g}, we obtain:
      \[
       \E \sca{\ell}{I_{n}}  = \sca{\ell}{\bar{f}_n} = \inv{ n} \tr{R_n(\ell)}, 
      \]
      where $\tr{A}$ denotes the trace of the matrix $A$, and $R_n(\ell) = \Sigma^{\inv{2}}_n T_n(\ell) \Sigma^{\inv{2}}_n$. Therefore the difference $\sca{\ell}{I_n-\bar{f}_n}$ takes the form:
      \[
        \sca{\ell}{I_n-\bar{f}_n} = \inv{ n} \left(\xi^T R_{n}(\ell) \xi - \tr{R_n(\ell)} \right).
      \]

      We  shall take $\ell=g_k-g_{k^*}$. For  this reason, we assume that
      $\ell$   is    even   and   $\norm{\ell}_\infty\leq    2K$.    Let
      $\eta=(\eta_i, 1  \leq i  \leq n)$ denote  the eigenvalues  of the
      symmetric  matrix $R_n(\ell)$,  with $\eta_1$  having the  largest
      absolute     value.     Similarly     to     Lemma    4.2.      of
      \cite{bigot2009adaptive}, we have that for all $a>0$:
\begin{align}
\label{eq:bigot}
\nonumber
 \expp{-u}  
& \geq   \P\left(  \sca{\ell}{I_n-\bar{f}_n} \geq
  \frac{2\val{\eta_1}u}{ n} + \frac{2\norm{\eta} \sqrt{u}}{ n} \right)  \\
& \geq   \P\left(  \sca{\ell}{I_n-\bar{f}_n} \geq
  \frac{2\val{\eta_1}u}{ n} + \frac{\norm{\eta}^2}{ a n} + \frac{a u}{
  n} \right)  ,
\end{align}
where we  used for the  second inequality that $2\sqrt{vw}  \leq v/a+aw$
for all  $v,w,a > 0$.  Let us give  upper bounds for  $\val{\eta_1}$ and
$\norm{\eta}^2$.  We note  $\rho(A)$ for  $A  \in \R^{n  \times n}$  the
spectral radius of the matrix $A$.  Then by the well-known properties of
the spectral radius, we have that:
\[
\val{\eta_1}=\rho(R_n(\ell)) \leq \rho(\Sigma_n) \rho(T_n(\ell)) 
\]
We deduce from \reff{eq:upper-rho-T} that 
$\rho(\Sigma_n)=\rho(2\pi  T_n(f)) \leq  2\pi \norm{f}_\infty  \leq 
\exp(L)$
and $\rho(T_n(\ell)) \leq \norm{\ell}_\infty\leq 2K$.  Therefore we obtain:
\begin{equation} \label{eq:eta_1}
\val{\eta_1}  \leq 2 K \expp{L}. 
\end{equation}
As for $\norm{\eta}^2$, we have:
\begin{equation} 
\label{eq:eta^2}
\norm{\eta}^2 = \tr{R^2_n(\ell)} = \tr{(\Sigma_n
  T_n(\ell))^2} \leq \rho(\Sigma_n)^2 \tr{T^2_n(\ell)}  
\leq   \expp{2L} \, n\norm{\ell}^2_{L^2(h)},
\end{equation}
where we used \reff{eq:Tope-tr2} for the last inequality. Using \reff{eq:eta_1} and \reff{eq:eta^2} in \reff{eq:bigot} gives:
\begin{align*}
 \expp{-u} 
& \geq \P\left(  \sca{\ell}{I_n-\bar{f}_n} \geq \frac{4 K
  \expp{L} u}{ n} + \frac{ \expp{2L}   \norm{\ell}^2_{L^2(h)}}{ a } +
  \frac{a u}{  n} \right) \\ 
& \geq \P\left(  \sca{\ell}{I_n-\bar{f}_n} - \frac{\expp{-3K}}{4}
  \norm{\ell}^2_{L^2 (h)}\geq \frac{\beta u}{ n} \right), 
 \end{align*}
 where for the second inequality we  set $a= 4\exp(2L+3K)$. This proves
 \reff{eq:conc_ineq_spec}, thus \reff{eq:unif_exp_bound}.

\subsubsection*{Second step: Upper bound for the bias term $B_n(\gS-g_{k^*})$}
 We set $\ell_*=\gS-g_{k^*}$ and we have $\norm{\ell_*}_{2,r}\leq
 2K/\cc_r$. Let $(a_k)_{k\in \Z}$ be the corresponding Fourier
 coefficients, which are real as  $\ell_*$ is even. 
We decompose the  the bias term as follows:
\begin{equation} \label{eq:B_n_alt}
B_n(\ell_*) =  \sca{\bar{f}_n-f}{\ell_*} =
\sca{\bar{f}_{n,1}-f}{\ell_*}-  \sca{\bar{f}_{n,2}}{\ell_*}, 
\end{equation}
with $\bar{f}_{n,1}, \bar{f}_{n,2}$ given by, for $x \in [-\pi,\pi]$:
\[ 
\bar{f}_{n,1}(x)= \sum_{|j|<n} \frac{\gamma_j}{2\pi} \expp{ijx}
\quad \text{and} \quad 
\bar{f}_{n,2}(x)=  \inv{n} \sum_{|j|<n} \frac{|j|\gamma_j}{2\pi} \expp{ijx}. 
\]
For the first term of the right hand side of \reff{eq:B_n_alt} notice that:
\[
\bar{f}_{n,1}(x)-f(x) = - \sum_{|j|\geq n} \frac{\gamma_j}{2\pi}\expp{ijx}.
\]
We                              deduce                              that
$\sca{\bar{f}_{n,1}-f}{\ell_*}   =   \sca{\bar{f}_{n,1}-f}{\bar\ell_*}$,
with  $\bar  \ell_*= \sum_{|j|\geq  n}  a_j  \expp{ijx}$. Then,  by  the
Cauchy-Schwarz inequality, we get:
\[
\val{\sca{\bar{f}_{n,1}-f}{\bar \ell_*}} \leq
\norm{\bar{f}_{n,1}-f}_{L^2(h)} \norm{\bar\ell_*}_{L^2(h)}. 
\]
Thanks to Lemma \ref{lem:reg}, we get:
\[
\norm{\bar{f}_{n,1}-f}_{L^2(h)}^2  = \sum_{|j| \geq n}^\infty
\frac{\gamma_j^2}{4\pi^2}
 \leq \sum_{|j| \geq n}^\infty \frac{|j|^{2r}}{ n^{2r}
 }\frac{\gamma_j^2}{4 \pi^2}  \leq 
\inv{n^{2r}} \{f\}_{2,r}^2 \leq 
\inv{n^{2r}} \norm{f}_{2,r}^2 \leq 
\frac{C(r,L)^2}{n^{2r}} \cdot
\]
This gives $\norm{\bar{f}_{n,1}-f}_{L^2(h)}\leq C(r,L) n^{-r}$. 
Similarly, we have $\norm{\bar\ell_*}_{L^2(h)} \leq  {n^{-r}} \{\ell_*\}_{2,r}\leq  {n^{-r}} \norm{\ell_*}_{2,r} \leq 
 2 K n^{-r}/\cc_r$. 
We deduce that:
\begin{equation} \label{eq:bias_bound_1}
\val{\sca{\bar{f}_{n,1}-f}{\bar \ell_*}} \leq  \frac{2KC(r,L)}{\cc_r}\,  n^{-2r}.
\end{equation}

For the second term on the right hand side of \reff{eq:B_n_alt}, we have:
\[
\sca{\bar{f}_{n,2}}{\ell_*} = 
\inv{ n}  \sum_{ |j|<n}  \frac{|j|\gamma_j}{2\pi}  a_j.
\]
Using the Cauchy-Schwarz inequality and then Lemma \ref{lem:reg}, we get as $r > 1/2$:
\begin{equation} 
\label{eq:bias_bound_2}
 \val{\sca{\bar{f}_{n,2}}{\ell_*}}
\leq \inv{n}   \{\ell_*\}_{2,1/2}  \{f\}_{2,1/2}
\leq \inv{n}   \norm{\ell_*}_{2,r}  \norm{f}_{2,r}
\leq \frac{2KC(r,L)}{\cc_r} n^{-1}.
\end{equation}
Therefore combining \reff{eq:bias_bound_1} and \reff{eq:bias_bound_2},
we obtain the following upper bound for the bias: 
 \begin{equation} \label{eq:bias_bound}
 \val{B_n(\ell_*)} \leq \frac{4KC(r,L)}{\cc_r} n^{-1}.
 \end{equation}

 \subsubsection*{Third  step: Conclusion}  Use 
 \reff{eq:unif_exp_bound}       and        \reff{eq:bias_bound}    in 
 \reff{eq:theo_spectral} to get the result.
\end{proof}

    \section{Lower bounds} \label{sec:lower}
    
    In this section we show that the aggregation procedure given in Section \ref{sec:aggr_method} is optimal by giving a lower bound corresponding to the upper bound of Theorem \ref{theo:aggreg_density} and \ref{theo:aggreg_spectral} for the estimation of the probability density function as well as for the spectral density.
    
\subsection{Probability density estimation} \label{sec:lower_D}
In this section we suppose that the reference density is the uniform
distribution on $[0,1]^d$: $h=\ind_{[0,1]^d}$. 

\begin{rem}
  If the reference density is not the uniform distribution on $[0,1]^d$,
  then    we   can    apply   the    Rosenblatt   transformation,    see
  \cite{rosenblatt1952}, to reduce the problem to this latter case. More
  precisely, according to \cite{rosenblatt1952},  if the random variable
  $Z$ has  probability density $h$, then  there exists two maps  $T$ and
  $T^{-1}$   such   that   $U=T(Z)$   is  uniform   on   $[0,1]^d$   and
  a.s.  $Z=T^{-1}(U)$.  Then if  the  random  variable $X$  has  density
  $f=\exp   (g)\,    h$,   we    deduce   that   $T(X)$    has   density
  $f^T=\exp ( g\circ  T^{-1}) \ind_{[0,1]^d}$. Furthermore, if  $f_1$ and
  $f_2$ are two  densities (with respect to the  reference density $h$),
  then we have $\kl{f_1}{f_2}=\kl{f_1^T}{f_2^T}$.
\end{rem}

We  give  the  main  result  of this  Section.  Let  $\P_f$  denote  the
probability measure when  $X_1, \hdots  , X_n$ are  i.i.d. random  variable with
density $f$.

\begin{prop} \label{prop:aggreg_low} 
Let  $N \geq 2$, $L>0$. Then there  exist $N$ probability densities  $(f_k, 1 \leq k  \leq N)$,
 with $f_k \in \cf^D(L)$ such that for all  $n  \geq 1$, $x \in \R^+$ satisfying:
\begin{equation}\label{eq:low_bound_cond}
\frac{\log(N)+x}{n} < 3 \left(1-\expp{-L}\right)^2,
\end{equation}
 we have:
 \[
 \inf_{\hat{f}_n} \sup_{f \in \cf^D(L)} \P_f\left( \kl{f}{\hat{f}_n} -
   \min_{1 \leq k \leq N} \kl{f}{f_k} \geq
   \frac{\beta'\left(\log(N)+x\right)}{n} \right) \geq \inv{24}
 \expp{-x}, 
\]
with  the infimum  taken over  all estimators  $\hat{f}_n$ based  on the
  sample $X_1,  \hdots,X_n$,  and $\beta' = 2^{-17/2}/3$.
 \end{prop}

In the following proof, we shall  use the  Hellinger distance  which is
defined as follows. For two non-negative integrable functions $p$ and $q$,
 the Hellinger distance $H(p,q)$ is defined as:
 \[
H(p,q)=\sqrt{\int \left( \sqrt{p}-\sqrt{q} \right)^2 }. 
\]
A well  known property of  this distance is  that its square  is smaller
then the Kullback-Leibler divergence defined by \ref{eq:KL}, that is  for all
non-negative integrable functions $p$ and $q$, we have:
\[
H^2(p,q) \leq \kl{p}{q}.   
\]

\begin{proof}
Since the probability densities $(f_k, 1 \leq k \leq N)$ belongs to $\cf^D(L)$, we have:
\begin{multline*}
  \inf_{\hat{f}_n} \sup_{f \in \cf^D(L)} \P_f\left( \kl{f}{\hat{f}_n} -
  \min_{1 \leq k \leq N} \kl{f}{f_k} \geq 
  \frac{\beta'\left(\log(N)+x\right)}{n} \right)  \\ 
\begin{aligned}
& \geq  \inf_{\hat{f}_n} \max_{1\leq k \leq N} \P_{f_k}\left(
  \kl{f_k}{\hat{f}_n} \geq   \frac{\beta'\left(\log(N)+x\right)}{n} \right)
\\  
& \geq  \inf_{\hat{f}_n} \max_{1\leq k \leq N} \P_{f_k}\left(
  H^2(f_k,\hat{f}_n) \geq  \frac{\beta'\left(\log(N)+x\right)}{n} \right). 
\end{aligned} 
\end{multline*}

For the choice of  $(f_k, 1 \leq k \leq N)$, we  follow the choice given
in  the proof  of Theorem  2 of  \cite{lecue2006lower}. Let  $D$ be  the
smallest   positive   integer   such   that   $2^{D/8}   \geq   N$   and
$\Delta=\{0,1\}^D$.   For   $0 \leq  j   \leq  D-1$,  $s  \in   \R$,  we
set:
\[
\alpha_j(s)=\frac{T}{D}\ind_{(0,\inv{2}]}(Ds- j)-\frac{T}{D}\ind_{(\inv{2},1]}(Ds-j),
\]
where  $T$   verifies  $0  <   T  \leq  D(1-\expp{-L})$.  Notice the
support of the function $\alpha_j$ is $(j/D, (j+1)/D]$. Then   for  any
$\delta=(\delta_1,\hdots, \delta_D) \in \Delta$, the function $f^\delta$
defined by:
\[
        f^\delta(y)=1+\sum_{j=0}^{D-1} \delta_j \alpha_j(y_1), \quad
        y=(y_1,\hdots,y_d) \in [0,1]^d, 
\]
is     a    probability     density      function     with
$\expp{L} \geq 1+T/D \geq f \geq  1-T/D \geq \expp{-L}$. This implies that
$f^\delta  \in  \cf^D(L)$.  As  shown  in  the  proof  of Theorem  2  in
\cite{lecue2006lower},   there   exists    $N$   probability   densities
$(f_k,1 \leq k  \leq N)$ amongst $\{f^\delta, \delta  \in \Delta\}$ such
that for any $i\neq j$, we have:
\[
        H^2(f_i,f_j) \geq \frac{8^{-3/2}T^2}{4D^2},
\]
and $f_1$ can be chosen to be the density of the uniform distribution on
$[0,1]^d$.  Recall  the  notation  $p^{\otimes n}$  of  the  $n$-product
probability density  corresponding to the probability  density $p$. Then
we also have  (see the proof of 
Theorem 2 of \cite{lecue2006lower})   for all $1 \leq i \leq N$:
     \[
       \kl{f^{\otimes n}_i}{f^{\otimes n}_1} \leq \frac{nT^2}{D^2}\cdot
     \]
     Let us take $T=D \sqrt{(\log(N)+x)/3n}$, so that with condition
     \reff{eq:low_bound_cond} we indeed have $T \leq D(1-\expp{-L})$. With
     this choice, and the defintion of $\beta'$, we have for $1 \leq i \neq j \leq N$ 
     \[
       H^2(f_i,f_j) \geq  4\frac{\beta' \left( \log(N)+x\right)}{n}
\quad\text{and}\quad 
      \kl{f^{\otimes n}_i}{f^{\otimes n}_1} \leq \frac{\log(N)+x}{3} \cdot
     \]
     Now we apply Corollary 5.1 of \cite{bellec2014optimal} with $m=N-1$
     and with the squared Hellinger distance instead of the  $L^2$
     distance to get that  for any estimator $\hat{f}_n$:
\[
  \max_{1\leq k \leq N} \P_{f_k}\left( H^2(f_k,\hat{f}_n) \geq
    \frac{\beta'\left(\log(N)+x\right)}{n} \right) 
\geq \inv{12}\min\left(1,(N-1)\expp{- (\log(N)+x)}\right) 
\geq \inv{24} \expp{-x}.
\]
This  concludes the proof.     
\end{proof}

\subsection{Spectral density estimation}
    
In this section we give a lower bound for aggregation of spectral density estimators.       
  Let $\P_f$  denote  the
probability measure when $(X_n)_{n\in \Z}$ is  a centered Gaussian sequence with
spectral  density  $f$.  Recall  the   set  of  positive  even  function
$\cf_r^S(L) \subset \cg$ defined by \reff{eq:def-FS} for $r \in \R$.

\begin{prop} \label{prop:spectr_low} 
       Let $N \geq 2$, $r>1/2$, $L>0$. There exist a constant $C(r,L)$ and $N$ spectral densities $(f_k, 1 \leq k \leq N)$
      belonging to $\cf_r^S(L)$ such that for all $n \geq 1$, $x \in \R^+$ satisfying:
      \begin{equation}\label{eq:low_bound_cond_spec}
         \frac{\log(N)+x}{n} < \frac{C(r,L)}{\log(N)^{2 r}}
      \end{equation}
      we have:
      \begin{equation} \label{eq:lower_spec_est}
        \inf_{\hat{f}_n} \sup_{f \in \cf_r^S(L)} \P_f\left(
          \kl{f}{\hat{f}_n} - \min_{1 \leq k \leq N} \kl{f}{f_k} \geq
          \frac{\beta' \left(\log(N)+x\right)}{n} \right) \geq \inv{24}
        \expp{-x}, 
      \end{equation}
with the infimum taken over all estimators $\hat{f}_n$ based on the
sample sequence $X=(X_1, \hdots, X_n)$, and $\beta' = 8^{-5/2}/3$.
\end{prop}
    
\begin{proof}
Similarly to the proof of Proposition \ref{prop:aggreg_low}, the left hand side of \reff{eq:lower_spec_est} is greater than: 
\[
       \inf_{\hat{f}_n} \max_{1\leq k \leq N} \P_{f_k}\left(
         H^2(f_k,\hat{f}_n) \geq  \frac{\beta'\left(\log(N)+x\right)}{n}
       \right). 
\]

We shall choose a set of spectral densities $(f_k, 1 \leq k \leq N)$
similarly as  in the  proof of  Proposition \ref{prop:aggreg_low}  such that
$f_k  \in  \cf_r^S(L)$. Let   us   define
$\phi:  [0,\pi]  \rightarrow \R$ as,  for
$x\in [0,\pi]$:
\[
       \phi(x)=  \zeta(x) \ind_{\left[0,\pi/2 \right]}(x) - 
       \zeta(x) \ind_{\left[\pi/2,\pi\right]}(x)
\quad\text{with}\quad  \zeta(x) = \expp{-1/{x\left(\frac{\pi}{2}-x\right)}}.
    \]
We have  that  $\phi \in C^\infty(\R)$ and:
\begin{equation} \label{eq:phi_prop}
\norm{\phi}_\infty =  \expp{-16/\pi^2},  
\quad  \int_0^\pi \phi = 0. 
\end{equation}
 Let  $D$  be  the  smallest integer  such  that
$2^{D/8}    \geq    N$    and   $\Delta=\{0,1\}^D$.     For  
$1 \leq j  \leq D$, $x \in [0,\pi]$, let  $\bar{\alpha}_j(x)$ be defined
as:
\[
\bar{\alpha}_j(x)=\phi(Dx-(j-1)\pi),
\]
and for  any $\delta=(\delta_1,\hdots, \delta_D) \in \Delta$ and $s\geq 0$, let the
function $f^\delta_s$ be defined by:
\begin{equation} \label{eq:def_f_delta}
        2\pi\, f^\delta_s (y)=1 +s\sum_{j=1}^D \delta_j
          \bar{\alpha}_j(\val{y}) , \quad y \in [-\pi,\pi].
\end{equation}
Since $\int_0^\pi \phi =0$, we get:
\begin{equation}
   \label{eq:majo-fd}
\inv{2\pi} \int_{-\pi}^\pi f_s^\delta(x) \, dx =1 \quad \text{and} \quad 1- s\norm{\varphi}_\infty \leq  2\pi f_s^\delta\leq 1+s\norm{\varphi}_\infty .
\end{equation}
We assume that $s \in [0,1/2]$, so that $2\pi f_s^\delta \geq 1/2$. Let us denote $g^\delta_s  = g_{f^\delta_s}=\log(2\pi f^\delta_s)$. 
We first give upper bounds for $\norm{(g^\delta_s)^{(p)}}_{L^2(h)}$ with $p \in \N$. 

For $p=0$, we have by \reff{eq:majo-fd} :
\begin{equation} \label{eq:gds_L2_up}
  \norm{g^\delta_s}_{L^2(h)} \leq \log\left(\inv{1-s\norm{\phi}_\infty}\right) \leq \frac{s\norm{\phi}_\infty}{1-s\norm{\phi}_\infty} \leq 2s.
\end{equation}

For $p \geq 1$, we get by Fa\`{a} di Bruno's formula that:
\begin{equation} \label{eq:faa_di_bruno}
  \norm{(g^\delta_s)^{(p)}}_{L^2(h)} = \left\| \sum_{k \in \ck_p} \frac{p!}{k_1!k_2! \hdots k_p!} \frac{(-1)^{\bar{k}+1} \bar{k}!}{(2\pi f^\delta_s)^{\bar{k}}}  \prod_{\ell=1}^p \left(\frac{(2 \pi f^\delta_s)^{(\ell)}}{\ell!} \right)^{k_\ell}   \right\|_{L^2(h)},
\end{equation}
with $\ck_p = \{ k=(k_1, \hdots, k_p) \in \N^p; \sum_{\ell=1}^p \ell k_\ell = p \}$ and $\bar{k} = \sum_{\ell=1}^p k_\ell$.  The $\ell$-th derivative of $2\pi f^\delta_s$ is given by, for $y \in [0,\pi]$:
\[
  (2\pi f^\delta_s(y))^{(\ell)} = s D^\ell \sum_{j=1 }^D \delta_j \phi^{(\ell)}(Dy-(j-1)\pi).
\]
Therefore we have the following bound for this derivative:
\begin{equation*} \label{eq:up_bound_fds_der}
   \norm{(2\pi f^\delta_s(y))^{(\ell)}}_\infty \leq s D^\ell \norm{\phi^{(\ell)}}_\infty. 
\end{equation*}
From $\phi \in C^\infty(\R)$, we deduce that $\norm{\phi^{(\ell)}}_\infty$ is finite for all $\ell \in \N^*$. Since $s \in [0,1/2]$ and  $2\pi f^\delta_s \geq 1-s\norm{\phi}_\infty \geq 1/2$, there exists a constant $\bar{C}_{p}$ depending on $p$ (and not depending on $N$), such that :
\begin{equation} \label{eq:gds_(p)_up}
\norm{(g^\delta_s)^{(p)}}_{L^2(h)} \leq s \bar{C}_p D^p  \leq s \bar{C}_p \frac{16^p}{\log(2)^p} \log(N)^p.
\end{equation}

In order to have $f^\delta_s    \in     \cf_r^S(L)$,
we need to ensure that $\norm{g^\delta_s}_{2,r} \leq L/\cc_r$. For $r\in \N^*$, we have:
\[
   \norm{g^\delta_s}_{2,r} = \sqrt{\norm{g^\delta_s}^2_{L^2(h)}+ \norm{(g^\delta_s)^{(r)}}^2_{L^2(h)}}.
\]
Therefore if $s \in [0,s_{r,L}]$ with $s_{r,L} \in [0,1/2]$ given by:
\[
   s_{r,L} = \log(N)^{-r} \bar{C}_{r,L}, \quad \text{ with } \quad \bar{C}_{r,L} = \min\left(\frac{\log(2)^r}{2}, \frac{\log(2)^r L}{\sqrt{8}\cc_r}, \frac{\log(2)^r L}{ \sqrt{2}  \cc_r 16^r \bar{C}_r} \right),
\]
then by \reff{eq:gds_L2_up} and \reff{eq:gds_(p)_up} we get:
\[
   \norm{g^\delta_s}_{2,r} \leq \sqrt{ \frac{L^2}{2\cc_r^2}  + \frac{L^2}{2\cc_r^2}} = \frac{L}{\cc_r} \cdot
\]

Let $\lceil r \rceil$  and $\lfloor r \rfloor$ denote 
the unique integers such that $\lceil r \rceil -1 < r \leq \lceil r \rceil$ and 
$\lfloor r \rfloor  \leq r < \lfloor r \rfloor+1$. 
For $r \notin \N^*$, H\"{o}lder's inequality yields:
\begin{align*}
   \norm{g^\delta_s}_{2,r} & = \sqrt{ \norm{g^\delta_s}^2_{L^2(h)}+ \left\{g^\delta_s\right\}^2_{2,r}} \\
                           & \leq  \sqrt{ \norm{g^\delta_s}^2_{L^2(h)}+ \left\{g^\delta_s\right\}^{2(r-\lfloor r \rfloor)}_{2,\lceil r\rceil} \left\{g^\delta_s\right\}^{2(\lceil r \rceil -  r )}_{2,\lfloor r\rfloor}} \\
                           & =\sqrt{ \norm{g^\delta_s}^2_{L^2(h)}+ \norm{(g^\delta_s)^{(\lceil r\rceil)}}^{2(r-\lfloor r \rfloor)}_{L^2(h)} \norm{(g^\delta_s)^{(\lfloor r\rfloor)}}^{2(\lceil r \rceil -  r )}_{L^2(h)}}.
\end{align*}
Using \reff{eq:gds_(p)_up} and \reff{eq:gds_(p)_up} with $p=\lceil r \rceil$ and  $p=\lfloor r \rfloor$, we obtain:
\[
  \norm{(g^\delta_s)^{(\lceil r\rceil)}}^{2(r-\lfloor r \rfloor)}_{L^2(h)} \norm{(g^\delta_s)^{(\lfloor r\rfloor)}}^{2(\lceil r \rceil -  r )}_{L^2(h)} \leq s^2 \bar{C}^{2(r-\lfloor r \rfloor)}_{\lceil r\rceil} \bar{C}^{2(\lceil r \rceil -  r )}_{\lfloor r\rfloor} \frac{16^{2r}}{\log(2)^{2r}} \log{N}^{2r}. 
\]
Hence if $s \in [0,s_{r,L}]$ with $s_{r,L} \in [0,1/2]$ given by:
\[
     s_{r,L} = \log(N)^{-r} \bar{C}_{r,L}, \quad \text{ with } \quad \bar{C}_{r,L} = \min\left(\frac{\log(2)^r}{2}, \frac{\log(2)^r L}{\sqrt{8}\cc_r}, \frac{\log(2)^r L}{ \sqrt{2}  \cc_r 16^r \bar{C}^{r-\lfloor r \rfloor}_{\lceil r\rceil} \bar{C}^{\lceil r \rceil -  r }_{\lfloor r\rfloor}} \right),
\]
we also have $\norm{g^\delta_s}_{2,r} \leq L/\cc_r$, providing $f^\delta_s    \in     \cf_r^S(L)$.

Mimicking  the proof of Theorem 2 in \cite{lecue2006lower} and omitting
the details, we first 
obtain (see last inequality of p.975 in \cite{lecue2006lower}) that for
$\delta, \delta'\in \Delta$:
\[
 H^2\left(f^{\delta}_s,f^{\delta'}_s\right) \geq   8^{-3/2} 
 \frac{\sigma(\delta,\delta')}{D} \frac{2 }{\pi} s^2 \int_0 ^\pi \varphi^2,
\]
with $\sigma(\delta,\delta')$  the Hamming distance between
$\delta$ and $\delta'$, and then deduce that there   exist  
$(\delta^{k},1 \leq k  \leq N)$ in $\Delta$ with $\delta^1=0$ such
that for any   $1 \leq i\neq j \leq N$ and $s\in [0, s_{r,L}]$, we have (see first
inequality of p.976 in \cite{lecue2006lower}):
\[
        H^2(f_s^{\delta^i},f_s^{\delta^j}) \geq \frac{2 \cdot 8^{-5/2}}{\pi} s^2 \int_0
        ^\pi \varphi^2.
\]
Notice  $f_s^{\delta^1}=f_s^{0}=h$ is  the density of the uniform distribution on
$[-\pi,\pi]$. 

With a  slight abuse  of notation,  let us denote  by $\rP_f$  the joint
probability    density    of     the    centered    Gaussian    sequence
$X=(X_1,\hdots,X_n)$ corresponding  to the spectral density  $f$. Assume
$X$  is standardized   (that is  $\Var(X_1)=1$),  which implies  $\int f=1$.   Let
$\Sigma_{n,f}$   denote   the   corresponding  covariance   matrix.    Since
$h=(1/2\pi)   \ind_{[-\pi,  \pi]}$,   we  have   $\Sigma_{n,h}=  \ci_n$   the
$n\times n$-dimensional identity  matrix. We compute:
\begin{align*}
        \kl{\rP_f}{\rP_h} 
& = \int_{\R^n} \rP_f(x) \log\left(\frac{\rP_{f}(x)}{\rP_h(x)} \right)
  \, dx \\
& = \int_{\R^n} \rP_f(x) \log\left( \inv{\sqrt{\det( \Sigma_{n,f}) }}
  \exp\left(-\inv{2} x^T \left(\Sigma_{n,f}^{-1} - \ci_n \right) x \right)
  \right) \, dx \\ 
& =  - \frac{1}{2}\log\left(\det( \Sigma_{n,f}) \right) - \inv{2}
  \rE_{f}\left[X^T \left(\Sigma_\delta^{-1} - \ci_n \right) X\right]. 
     \end{align*}
The expected value in the previous equality can be written as:
\[
        \rE_{f}\left[X^T \left(\Sigma_{n,f}^{-1} - \ci_n \right) X\right] 
= \tr{\left(\Sigma_{n,f}^{-1} - \ci_n \right)\rE_{f}[X^T X]} 
=\tr{\ci_n-\Sigma_{n,f}}=0,
\]
where for the last equality, we  used that the Gaussian random variables
are                 standardized.                 This                 yields
$\kl{\rP_f}{\rP_h}   =   -   \frac{1}{2}\log\left(\det(\Sigma_{n,f}) )
\right)$. 
We can use this  last equality for
$f=f^\delta_s$ since  $\int f^\delta_s=1$  thanks to  \reff{eq:phi_prop}, and
obtain:
\[
\kl{\rP_{f_s^\delta}}{\rP_{f_s^0}} =  - \frac{1}{2}\log\left(\det(
  \Sigma_{n,f_s^\delta}) \right).  
\]
Notice that for $s \in [0,s_{r,L}]$, we have $3/2 \geq 1+s \norm{\varphi}_\infty \geq  2 \pi f_s^\delta \geq  1 - s \norm{\varphi}_\infty \geq 1/2 $
thanks to \reff{eq:majo-fd} and \reff{eq:phi_prop}. Therefore we have:
\begin{equation} \label{eq:kl_fdelta}
\kl{\rP_{f_s^\delta}}{\rP_{f_s^0}}\leq \frac{n}{2} \norm{2\pi f_s^\delta-1}^2_{L^2(h)} \leq \frac{n}{2} \frac{s^2}{\pi} \int_0^\pi \phi^2,
\end{equation}
where we used $\Sigma_{n, f_s^\delta}=T_n(2\pi f_s^\delta)$ and Lemma \ref{lem:Tn_logdet} with $\ell=2\pi f_s^\delta$ for the first inequality, and \reff{eq:def_f_delta} for the second inequality.
We set:
\[
C(r,L) = \frac{3 \bar{C}^2_{r,L}\int_0^\pi \phi^2}{ 2  \pi } \quad \text{ and }  \quad s = \sqrt{\frac{2\pi}{3\int_0^\pi \phi^2}} \sqrt{\frac{\log(N)+x}{n}},
 \]
 so that \reff{eq:low_bound_cond_spec} holds for $s \in [0,s_{r,L}]$. We obtain for all $\delta^1,\delta^2 \in \bar{\Delta}$, $\delta \in \Delta$:
     \[
        H^2\left(f^{\delta_1}_s,f^{\delta_2}_s\right) \geq 4 \frac{\beta'(\log(N)+x)}{n} \quad \text{ and }  \quad \kl{\rP_{f_s^\delta}}{\rP_{f_s^0}} \leq \frac{\log(N)+x}{3} \cdot
     \]      
We conclude the proof as in the end of the proof of Proposition \ref{prop:aggreg_low}.

    \end{proof}

\section{Appendix}    
\subsection{Results on Toeplitz matrices}
\label{sec:toeplitz}

Let  $\ell\in L^1(h)$ be a real function with $h=1/(2\pi) \ind_{[-\pi,\pi]}$. We define the corresponding
Toeplitz matrix $T_n(\ell)$ of size $n\times n$  of its Fourier coefficients by:
\begin{equation}
   \label{eq:toeplitz}
[T_n(\ell)]_{j,k}=\inv{2\pi}\int_{-\pi}^\pi \ell(x) \expp{i(j-k)x} \, dx
 \quad \text{ for } 1 \leq j,k \leq n.
\end{equation}
Notice that $T_n(\ell)$ is Hermitian. It is also real if $\ell$ is
even. Recall that $\rho(A)$ denotes the spectral density of the matrix $A$.

\begin{lem} Let $\ell \in L^2(h)$ be a real function.
\begin{enumerate}
\item All the  eigenvalues of  $T_n(\ell)$ belong to $[\min \ell, \max
  \ell]$. In particular,  we have the following 
  upper bound on the spectral radius $\rho(T_n(\ell))$ of $T_n(\ell)$:
\begin{equation}
   \label{eq:upper-rho-T}
\rho(T_n(\ell))\leq  \norm{\ell}_\infty .
\end{equation}
\item For the trace of $T_n(\ell)$ and $T^2_n(\ell)$, we have: 
\begin{equation}
   \label{eq:Tope-tr2}
\tr{T_n(\ell)}=\frac{n}{2\pi} \int_{-\pi}^\pi \ell(x)  \, dx
\quad\text{and}\quad \tr{T^2_n(\ell)}  \leq n \norm{\ell}^2_{L^2(h)}.
\end{equation}

\end{enumerate}

\end{lem}
\begin{proof}
For Property (1), see Equation (6) of Section 5.2
  in \cite{grenander1958toeplitz}. For Property (2), the first part is
  clear and for the second part, see Lemma 3.1 of
  \cite{davies1973asymptotic}. 
  
\end{proof}

We shall use the following elementary result.

\begin{lem} \label{lem:Tn_logdet}
Let $\ell\in L^2(h)$ such that $\int \ell h =1$ and $\ell(x) \in [1/2,3/2]$, then  we have:
\begin{equation}
   \label{eq:majo-det-T}
\log\left(\det(
  T_n(\ell)) \right)\geq -n \norm{\ell-1}^2_{L^2(h)}.
\end{equation}
\end{lem}

\begin{proof}
Notice that by Property (1), the eigenvalues
  $(\nu_i, 1 \leq i \leq n)$ of $T_n(\ell)$ verify $\nu_i \in [1/2,3/2]$.
  For $t \in [-1/2,1/2]$, we have $\log(1+t)\geq t- t^2$, giving that:  
\[
\log\left(\det(
 T_n(\ell)) \right)
=\sum_{i=1}^n \log(\nu_i)
\geq \sum_{i=1}^n (\nu_i -1) -(\nu_i-1)^2
=  - \tr{T^2_n(\ell-1)}
 \geq -n\norm{\ell-1}^2_{L^2(h)},
\]
where we used that $T_n(\ell-1) = T_n(\ell)-\ci_n$ for the second equality and Property (2) for 
the second inequality. 
\end{proof}

\subsection{Proof of Lemma \ref{lem:reg}}
 \label{sec:preuve-exp-S}

The next Lemma is inspired by the work of \cite{DiNezza2012521} 
on fractional Sobolev spaces. 
For $r\in (0,1)$ and $\ell\in L^2(h)$, we define:
\[
I_r(\ell)=\inv{2\pi}\int_{[-\pi, \pi]^2} \frac{|\ell(x+y) -
  \ell(x)|^2}{|y|^{1+2r}}\, dx dy,
\]
where we set $\ell(z)=\ell( z  -2\pi)$ for $z\in (\pi, 2\pi]$ and $\ell(z)=\ell(z+2\pi)$ for $z \in [-2\pi,-\pi)$. 
\begin{lem} \label{lem:I_r_l}
 Let $r \in (0,1)$ and $\ell\in L^2(h)$. Then we have:
 \begin{equation}
   \label{eq:majo-Ir}
c_r \{\ell\}_{2,r}^2 \leq  I_r(\ell)  \leq
C_r \{\ell\}_{2,r}^2. 
\end{equation}
 
\end{lem}

\begin{proof}

Using the Fourier
representation of $\ell$, we get:
\[
I_r(\ell)=\sum_{k\in \Z} |a_k|^{2} \int_{-\pi}^{\pi} \frac{|1-
  \expp{iky}|^2}{|y|^{1+2r}} \, dy
= \sum_{k\in \Z} |k|^{2r} |a_k|^{2}
\int_{- |k|\pi}^{ |k| \pi} \frac{|1-
  \expp{iz}|^2}{|z|^{1+2r}} \, dz.
\]
For  $r\in (0,1)$ and $k\in \Z^*$, we have
\[
0<c_r:=\int_{- \pi}^{ \pi} \frac{|1-
  \expp{iz}|^2}{|z|^{1+2r}} \, dz\leq \int_{-|k|\pi}^{|k|\pi} \frac{|1-
  \expp{iz}|^2}{|z|^{1+2r}} \, dz \leq \int_\R \frac{|1-
  \expp{iz}|^2}{|z|^{1+2r}} \, dz=: C_r<+\infty.
\]
This yields \reff{eq:majo-Ir}.

\end{proof}

\subsubsection*{First step :  $r\in     (1/2,1)$}
Let    $r\in     (1/2,1)$    and set $L=\cc_r K$. Let   $f=\expp{g}$ with $g \in W_r$ such that $\norm{g}_{2,r}\leq
K$.   Thanks to \reff{eq:sobolev-ineq}, we    have
$\norm{g}_\infty          \leq         \cc_r K=L$.          Using         that
$|\expp{x} -\expp{y}|\leq \expp{L}|x-y|$ for $x,y\in [-L,L]$, we deduce
that:
\begin{equation}
   \label{eq:I(f)}
I_r(f)=I_r(\expp{g})\leq  \expp{2 L} I_r(g)
\quad\text{and}\quad
\norm{f}^2_{L^2(h)}\leq  \expp{2 L}.
\end{equation}
Using \reff{eq:majo-Ir} twice, we get:
\[
\norm{f}_{2,r}^2
\leq  \expp{2L} \left(1+ \frac{C_r}{c_r}\{g\}_{2,r}^2 \right)
\leq  \expp{2\cc_r K} \left(1+ \frac{C_r}{c_r}K^2 \right).
\]
Which proves the Lemma for $r\in (1/2,1)$.

\subsubsection*{Second step :  $r\in    \N^*$}
Let $r\in \N^*$.  For $\ell\in W_r$,  the $r$-th
derivative of $\ell$, say $\ell^{(r)}$,  exists in $L^2(h)$
and:
\[
\{\ell\}_{2,r}^2=  \norm{\ell^{(r)}}^2_{L^2(h)}
\quad\text{as well as}\quad
\norm{\ell}_{2,r}^2= \norm{\ell}^2_{L^2(h)}+ \norm{\ell^{(r)}}^2_{L^2(h)}.
\]
According to \reff{eq:sobolev-ineq}, we also get that for all $p\in \N$ with
$p<r$ we have $\norm{\ell^{(p)}}_\infty \leq \cc_{r-p}
\{\ell^{(r)}\}_{2,r}\leq \cc_{1}
\{\ell^{(r)}\}_{2,r}$.

Set $L=\cc_r K$. Let   $f=\expp{g}$ with $\norm{g}_{2,r}\leq
K$.  We have $\norm{g^{(p)}}_\infty \leq  \cc_1 K$ for all integer
$p<r$.   According  to  Leibniz's rule, we get that $f^{(r)}=g^{(r)} f +
P_r(g^{(1)}, \ldots, g^{(r-1)})f$, where $P_r$ is a polynomial function
of maximal degree $r$ such
that:
\begin{equation}
   \label{eq:maxPp}
\max_{x_1, \ldots, x_{r-1}\in [-\cc_{1}K, \cc_{1}K]} |P_r(x_1,
\ldots, x_{r-1})| \leq C_{r,1} K^r. 
\end{equation} 
for some finite constant $C_{r,1}$. We deduce that:
\[
\norm{f^{(r)}}_{L^2(h)}\leq  \expp{L} \norm{g^{(r)}}_{L^2(h)}
+\expp{L} C_{r,1} K^r.
\]
Then use that $\norm{f}_{L^2(h)}\leq  \expp{L}$ to get the Lemma for
$r\in \N^*$. 

\subsubsection*{Third step :  $r>1$, $r \not\in    \N^*$}
Let $r>1$ such that $r\not\in \N^*$. Set $p=\lfloor r \rfloor\in \N^*$ the
integer part of $r$ and $s=r-p\in (0,1)$. 
For $\ell\in W_r$,  the $p$-th
derivative of $\ell$, say $\ell^{(p)}$,  exists in $L^2(h)$
and:
\begin{equation}
   \label{eq:lr>1}
\{\ell\}_{2,r}^2=  \{\ell^{(p)}\}^2_{2,s}
\quad\text{as well as}\quad
\norm{\ell}_{2,r}^2= \norm{\ell}^2_{L^2(h)}+  \{\ell^{(p)}\}^2_{2,s}.
\end{equation}
Thanks to \reff{eq:majo-Ir} (twice) and the triangle inequality, we  have for all measurable function $t$:
\begin{equation}
   \label{eq:majo-lf}
c_s \{\ell t\}_{2,s}^2
\leq  I_s(\ell t) 
\leq  \norm{t}_\infty ^2 I_s(\ell) +J_s(\ell,t)
\leq   \norm{t}_\infty ^2 C_s\{\ell\}_{2,s}^2 +J_s(\ell,t), 
\end{equation}
with 
\[
J_s(\ell,t)= \inv{2\pi} \int_{[-\pi, \pi]^2} \ell(x)^2\, \frac{|t(x+y) -
  t(x)|^2}{|y|^{1+2s}}\, dx dy.
\]

Let $K>0$ and set $L=\cc_r K$. Let   $f=\expp{g}$ with $g \in W_r$ such that $\norm{g}_{2,r}\leq
K$. Following the proof of Lemma \ref{lem:I_r_l}, we first give an upper bound of $J_s(\ell,f)$ in this context under
the only condition that $\ell\in L^2(h)$. 
 Using         that
$|\expp{x} -\expp{y}|\leq \expp{L}|x-y|$ for $x,y\in [-L,L]$, we deduce
that:
\[
\int_{-\pi}^\pi\frac{|f(x+y) -
  f(x)|^2}{|y|^{1+2s}}\,  dy
\leq  \expp{2L} \int_{-\pi}^\pi\frac{|g(x+y) -
  g(x)|^2}{|y|^{1+2s}}\,  dy.
\]
Since a.e. $g(x)=\sum_{k\in \Z} a_k \expp{ikx}$, we deduce that:
\[
J_s(\ell, f) \leq  \frac{\expp{2L}}{2\pi}
\int_{-\pi}^\pi  dx\, 
\ell(x)^2\, \sum_{k,j\in \Z} |a_k| |a_j|
\int_{-\pi}^\pi\frac{|(1-\expp{iky})(1-\expp{-ijy})|}{|y|^{1+2s}}\,  dy.
\]
Let $\varepsilon\in (0,1/2)$ such that $s+\varepsilon\leq 1$. Since
$|1-\expp{ix}|\leq  2 |x|^{s+\varepsilon}$ for all $x\in \R$, we deduce
that:
\[
\int_{-\pi}^\pi\frac{|(1-\expp{iky})(1-\expp{-ijy})|}{|y|^{1+2s}}\,  dy
\leq C_{2,\varepsilon} |k|^{s+\varepsilon} |j|^{s+\varepsilon},
\]
for some constant $C_{2,\varepsilon}$ depending only on $\varepsilon$. 
Using  Cauchy-Schwarz inequality and the fact that
$r-s-\varepsilon>1/2$, we get: 
\[
\sum_{k\in \Z}  |k|^{s+\varepsilon} |a_k|
\leq  \cc_{r-s-\varepsilon}\{g\}_{2,r}.
\]
We deduce that:
\begin{equation}
   \label{eq:Jpf}
J_s(\ell, f) \leq  \expp{2L} \norm{\ell}^2_{L^2(h)} C_{2,\varepsilon}
\cc_{r-s-\varepsilon}^2 \{g\}_{2,r}^2.
\end{equation}
According  to  Leibniz's rule, we get that $f^{(p)}= \ell f+g^{(p)} f $
with $\ell=P_p(g^{(1)}, \ldots, g^{(p-1)})$. 
We get:
\begin{equation}
   \label{eq:Psf}
c_s \{\ell f\}^2_{2,s} 
\leq \norm{f}_\infty ^2 C_s\{\ell\}_{2,s}^2
+J_s(\ell,f)
\leq \expp{2L} C_s\{f\}_{2,s}^2
+ \expp{2L} \norm{\ell}^2_{L^2(h) } C_{2,\varepsilon}
\cc_{r-s-\varepsilon}^2 \{g\}_{2,r}^2,
\end{equation}
where we used   \reff{eq:majo-lf} for the first inequality and
\reff{eq:Jpf} for the latter. Then use \reff{eq:maxPp} with $r$ replaced
by $p$ to get that  
$\norm{\ell}_{L^2(h) }\leq  \norm{\ell}_\infty \leq  C_{p,1}K^p$.  
Notice also that:
\[
 \{f\}^2_{2,s}
\leq  \expp {2L}  \frac{C_s}{c_s} \{g\}^2_{2,s},
\]
using 
\reff{eq:majo-Ir} twice and  \reff{eq:I(f)} (with $s$ instead of $r$). 
We deduce that $ \{\ell f\}_{2,s} $ is bounded by a constant depending
only on $K$, $r$ and $\varepsilon$. 

The upper bound of $\{g^{(p)}f\}^2_{2,s} $ is similar. Using
\reff{eq:majo-lf} and \reff{eq:Jpf}, we get:
\[
c_s\{g^{(p)}f\}^2_{2,s} 
\leq  \norm{f}_\infty ^2 I_s (g^{(p)}) + 
J_s(g^{(p)}, f)
\leq  \expp{2L} C_s\{g^{(p)}\}^2_{2,s} + 
\expp{2L} \norm{g^{(p)}}^2_{L^2(h)} C_{2,\varepsilon}
\cc_{r-s-\varepsilon}^2 \{g\}_{2,r}^2.
\]
We deduce that  $ \{g^{(p)} f\}_{2,s} $, and thus  $f^{(p)}$, is bounded
by a constant depending only on $K$, $r$ and $\varepsilon$.
 Then use \reff{eq:lr>1} and that $\norm{f}_{L^2(h)}\leq \norm{f}_\infty \leq
\expp{L}$ to get the Lemma for $r>1$ and $r\not\in \N$. This concludes
the proof.

\bibliographystyle{abbrv}
\bibliography{aggreg_biblio}

\end{document}